\documentclass[reqno,final,11pt]{amsart}
\usepackage{natbib}  
\usepackage{fancyhdr} 
\usepackage{color} 
\usepackage{hyperref} 
\usepackage{graphicx} 
\usepackage{geometry}
\usepackage[utf8]{inputenc}
\usepackage[T1]{fontenc}
\usepackage{verbatim}

\usepackage{tikz}
\usetikzlibrary{calc,shapes,backgrounds,arrows}
\usepackage{amssymb}

\definecolor{aleacolor}{rgb}{0.16,0.59,0.78}

\hypersetup{
breaklinks,
colorlinks=true,
linkcolor=aleacolor,
urlcolor=aleacolor,
citecolor=aleacolor}

\renewcommand{\cite}{\citet}

\theoremstyle{plain}
\newtheorem{theorem}{Theorem}[section]                                          
\newtheorem{proposition}[theorem]{Proposition}                          
\newtheorem{lemma}[theorem]{Lemma}
\newtheorem{corollary}[theorem]{Corollary}

\theoremstyle{definition}
\newtheorem{definition}[theorem]{Definition}
\theoremstyle{remark}

\makeatletter \@addtoreset{equation}{section} \makeatother
\renewcommand\theequation{\thesection.\arabic{equation}}
\renewcommand\thefigure{\thesection.\arabic{figure}}
\renewcommand\thetable{\thesection.\arabic{table}}

\newcommand{\aleaIndex}[1]{\href{http://alea.impa.br/english/index_v18.htm}{\bf 18}}
\newcommand{\aleaDOI}[1]{\href{https://doi.org/10.30757/ALEA.v18-}{DOI: 10.30757/ALEA.v18-}}
\eheader{ALEA, Lat. Am. J. Probab. Math. Stat.}{\aleaIndex{18}}{2021}{1}{4}{\aleaDOI}
\elogo{\parbox[c]{180pt}{\includegraphics[width=170pt]{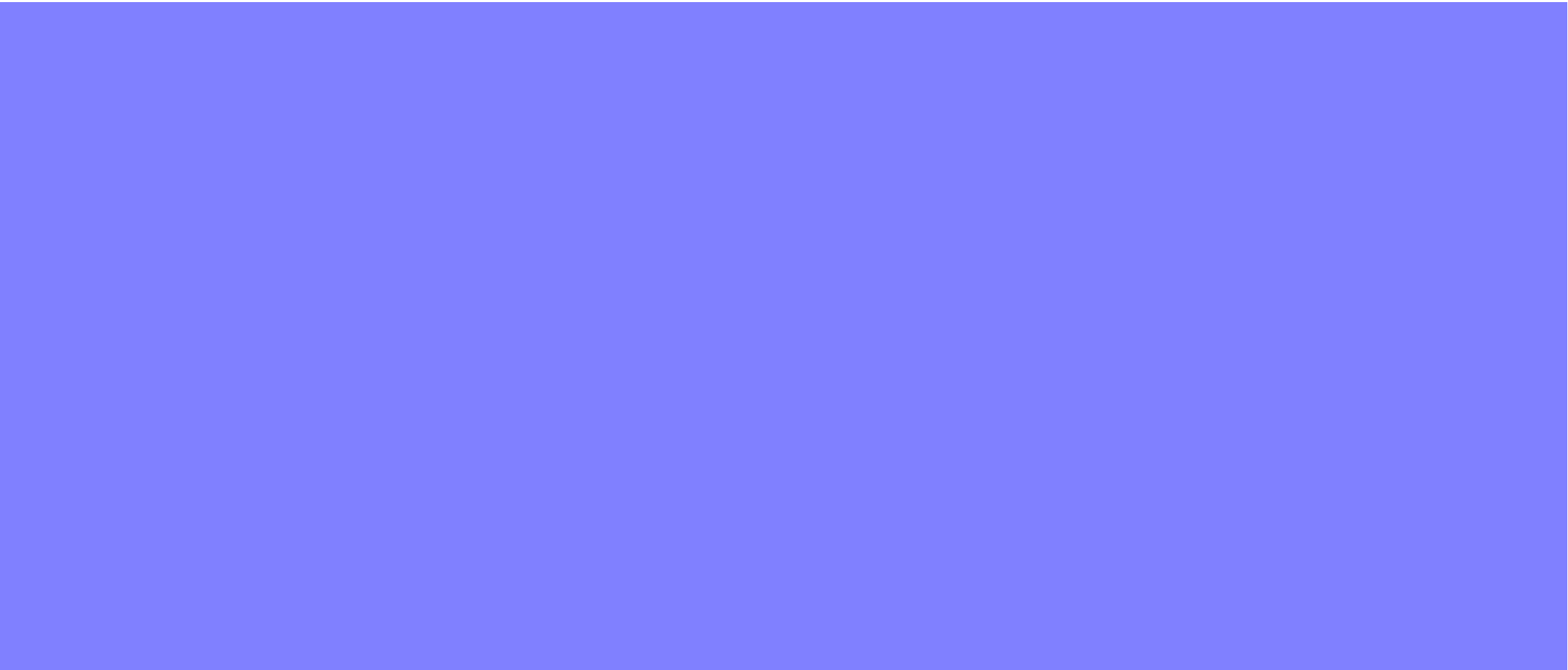}}} 

\usepackage{bbm}
\usepackage{enumitem}

\newcommand{\LL}{\mathcal{L}}
\newcommand{\R}{\ensuremath{\mathbb{R}}} 
\newcommand{\Lp}{\ensuremath{\mathbb{L}}} 
\newcommand{\D}{\ensuremath{\mathbb{D}}} 

\newcommand{\Lex}{\ensuremath{\mathbbm{e}}} 
\newcommand{\br}{\ensuremath{\mathbbm{b}}} 
\newcommand{\bra}{\ensuremath{\mathbbm{b}^{(a)}}} 
\newcommand{\brae}{\ensuremath{\mathbbm{b}^{(a_\varepsilon)}}}
\newcommand{\Dex}{\ensuremath{D_{\mathrm{ex}}}} 
\newcommand{\Hex}{\ensuremath{H_{\mathrm{ex}}}} 
\newcommand{\Kex}{\ensuremath{K_{\mathrm{ex}}}} 
\newcommand{\Aex}{\ensuremath{\mathcal{A}_{\mathrm{ex}}}} 
\newcommand{\Dbr}{\ensuremath{D_{\mathrm{br}}}}

\newcommand{\Dbra}{\ensuremath{D_\mathrm{br}^{(a)}}}
\newcommand{\Hbra}{\ensuremath{H_{\mathrm{br}}^{(a)}}}

\newcommand{\fup}{\ensuremath{f_\uparrow}}
\newcommand{\fupac}{\ensuremath{f_{\uparrow\mathrm{ac}}}}
\newcommand{\fups}{\ensuremath{f_{\uparrow\mathrm{sing}}}}
\newcommand{\fdo}{\ensuremath{f_\downarrow}}
\renewcommand{\d}{\ensuremath{\,\mathrm{d}}}
\newcommand{\dist}{\ensuremath{\mathrm{dist}}}

\newcommand{\E}[1]{\ensuremath{\mathbb{E} \left[#1 \right]}} 
\newcommand{\Prob}[1]{\ensuremath{\mathbb{P} \left(#1 \right)}} 
\newcommand{\bE}{\ensuremath{\mathbb{E}}} 
\newcommand{\bP}{\ensuremath{\mathbb{P}}} 
\newcommand{\1}[1]{\ensuremath{\mathbbm{1}_{\{ #1 \}}}} 
\newcommand{\Fil}{\mathcal{F}} 
\newcommand{\eps}{\ensuremath{\varepsilon}} 

\begin{document}

\title[A large deviation principle for the excursion of an $\alpha$-stable L\'evy process]{A large deviation principle for the normalized excursion of an $\alpha$-stable Lévy process without negative jumps}

\author{Léo Dort}
\address{Univ Lyon, INSA de Lyon, CNRS UMR 5208, Institut Camille Jordan, 21 avenue Jean Capelle, 69621 Villeurbanne, France}
\email{leo.dort@insa-lyon.fr} 
\urladdr{\href{https://sites.google.com/view/leo-dort}{https://sites.google.com/view/leo-dort}}

\author{Christina Goldschmidt}
\address{Department of Statistics and Lady Margaret Hall, University of Oxford, 24-29 St Giles', Oxford OX1 3LB}
\email{goldschm@stats.ox.ac.uk} 
\urladdr{\href{https://www.stats.ox.ac.uk/~goldschm/}{https://www.stats.ox.ac.uk/$\sim$goldschm/}} 

\author{Grégory Miermont}
\address{Unité Mathématiques Pures et Appliquées, CNRS UMR 5669, École
  Normale Supérieure de Lyon, 46 allée d'Italie, 69364 Lyon Cedex 07\\
and Institut Universitaire de France}
\email{gregory-miermont@ens-lyon.fr} 
\urladdr{\href{http://perso.ens-lyon.fr/gregory.miermont/}{http://perso.ens-lyon.fr/gregory.miermont/}}

\thanks{The second author was supported by EPSRC Fellowship EP/N004833/1.}

\begin{abstract}
We establish a large deviation principle for the normalized excursion
and bridge of an $\alpha$-stable L\'evy process without negative jumps, with
$1<\alpha<2$. Based on this, we derive precise asymptotics for the
tail distributions of  functionals of the normalized
excursion and bridge, in particular, the area and maximum
functionals. 
We advocate the use of the Skorokhod M1 topology, rather than
the more usual J1 topology, as we believe it is better suited to
large deviation principles for Lévy processes in general. 
\end{abstract}

\maketitle


\begin{center}
\includegraphics[scale=0.35]{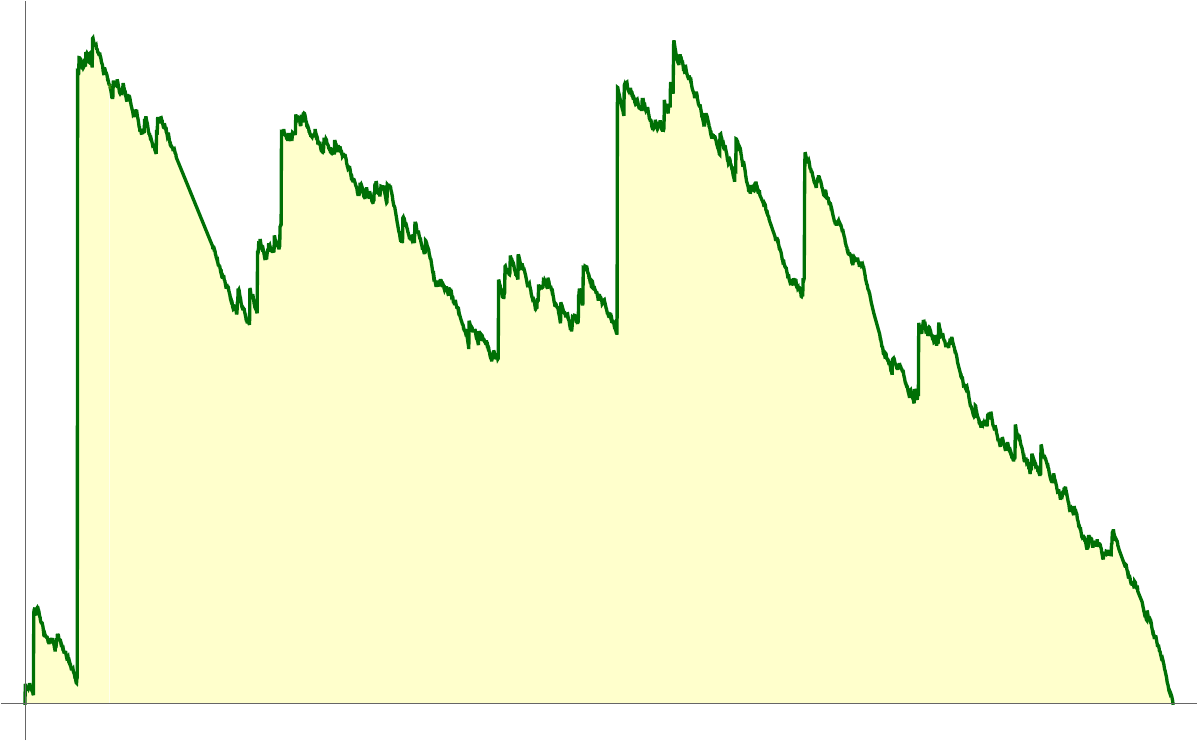} \par
\end{center}
\begin{minipage}{\textwidth}
\leftskip 3pc
\rightskip 3pc
Simulation of the area under the normalized excursion of a
  $\frac{4}{3}$-stable L\'evy process without negative jumps. 
\end{minipage}

\section{Introduction}

 Fix $\alpha \in (1,2)$. Let $L=(L_t,t\geq 0)$ be an $\alpha$-stable
 L\'evy process without negative jumps, having Laplace transform 
 $$ \E{\exp(-\lambda L_t)} = \exp(t\lambda^{\alpha})\, ,\qquad \lambda,t>0. $$ 
Our main goal in this paper is to establish a strong large deviation
principle for the normalized bridges and excursions of the process
$L$. 
This contributes to the rich study of functional large 
deviation principles ({\fontfamily{lmss}\selectfont LDP}) 
for Lévy processes and related 
processes, which has its roots in the classical theorem of 
Cramér and its extension to random walks, see Chapter 5.1 in 
\cite{Dembo-Zeitouni}, and references therein. For a Lévy process
$(X(t),t\geq 0)$, the natural setting is to consider the family of
renormalized processes
\begin{equation}
  \label{eq:22}
  X_T=(X(Tt)/T,0\leq t\leq 1). 
\end{equation}
The case where $X$ is Brownian
motion is addressed by a famous theorem of 
\cite{schilder66}, who showed a large deviation principle with speed $T$
in the space of continuous functions, where the rate function is the
Dirichlet energy. 
Other Lévy processes are addressed in the landmark
paper by 
\cite{Lynch}, which was extended in various directions in
particular by Borovkov, Mogul'ski\u{\i} and others (see, for instance,
\cite{mogulskii93,BoMo13,BoMo14,mogulskii17,KlMo19,KlLoMo20} and references therein).
However, the vast majority of the
results in the above references assume that the Lévy process has a
vanishing Gaussian part, as well as the Cramér condition that
$\E{e^{\lambda  X(1)}}<\infty$ for every $\lambda$ in a non-empty neighborhood of 
$0$. These conditions imply that the trajectories of the Lévy
process have finite variation almost surely, and that the law of $X(1)$ has
exponential tails. Various situations may occur when the Cramér
condition does not hold. The references \cite{gantert98,GaRaRe19}
consider the case of stretched exponential tails, while 
\cite{RhBlZw19} considers the situation where the
tail of $X(1)$ is regularly varying. In these cases,
large deviation principles hold with sublinear speeds, and even
logarithmic speed when the tails are regularly varying.

The case of stable Lévy processes with no negative jumps is in a sense
a boundary case of the works mentioned above, due to the asymmetric
nature of the tails of $L_1$: we have
$$\Prob{L_1>x}\asymp \frac{C_\alpha}{x^{\alpha}}\, ,\qquad 
\Prob{L_1<-x}\asymp \exp(-c_\alpha x^{\alpha'})\, ,\qquad
x\to\infty ,$$
where $\alpha'=\alpha/(\alpha-1)\in (2,\infty)$ is the conjugate exponent of
$\alpha$, $C_\alpha=-\frac{1}{\Gamma(1-\alpha)}$ and
$c_\alpha=(\alpha-1)/\alpha^{\alpha'}$.
Heuristically, although it is
``easy'' for the process to go up, it is ``costly'' for it to go
down, and large deviation probabilities may have different speeds depending on
whether the events involved allow the process to ``go down'' or not. 

However, considering bridges and excursions of such
processes is a way to root them at a given value at time $1$, which
prevents the process 
from ``going up too much''. This phenomenon is well-known and was already
exploited in \cite{ABDeJa13,Kortchemski}, in particular in the study
of the heights of
random trees. However, to our knowledge it has not been used to derive an
actual {\fontfamily{lmss}\selectfont LDP} result for stable 
excursions and bridges, and we fill this
gap in the present work. Note that {\fontfamily{lmss}\selectfont LDPs} 
for bridge-like random walks
and Lévy processes were considered under the Cramér condition in
\cite{BoMo13b}. 

Although we believe our results should have
extensions to a much larger class of Lévy process bridges and
excursions, we focus here on the particular case of stable
processes. In this case, precise estimates are known for the
transition densities and the entrance law of the excursion measure,
which allow us to provide a rather straightforward extension of the
method of proof presented by 
\cite{Serlet}. 
However, our proof differs from the latter in a crucial aspect. 
As is often the case in {\fontfamily{lmss}\selectfont LDP} theory, the choice of
an appropriate topology on the path space is an important matter. In \cite{Lynch}, the
authors derived a strong {\fontfamily{lmss}\selectfont LDP} for Lévy
processes as in \eqref{eq:22} in a ``weak'' topology, and observed that
the rate function is not good (in the sense that it
does not have compact level sets) in the natural Skorokhod J1
topology. In a series of papers,
Borovkov and Mogul'ski\u{\i} improved these results by
considering local versions of the {\fontfamily{lmss}\selectfont LDP} in the J1  topology, or by working
in the completion of the Skorokhod J1 metric. Unsurprisingly, similar questions
arise in our context. Indeed, since it is much less
costly for the process $L$ to go up rather than to go down, a similar
property also holds for its bridges and excursions, and this implies
that in the large deviation regimes considered in this paper (and in
stark contrast to \cite{RhBlZw19}), we cannot distinguish between
the situation where the process performs one big jump, or two jumps of
half the size at extremely close locations, precluding exponential
tightness in the J1 topology. 

Fortunately, Skorokhod introduced three other possible topologies,
called M1, J2 and M2, on spaces of
càdlàg functions, and M1 will
turn out to suit our purposes, with a very minor adaptation called M1' that was
already considered in \cite{BBRZ20} and in \cite{V21}, in the contexts
of {\fontfamily{lmss}\selectfont LDPs} for Lévy processes and random
walks with Weibull increments, and of a contraction principle for
random walk {\fontfamily{lmss}\selectfont LDPs}. We note that Mogul'ski\u{\i}
and others (\cite{BoMo14,mogulskii17}) also considered similar
topologies in
the context of processes satisfying the 
Cramér condition. The M1 topology was also used by 
\cite{obrien99}
to prove {\fontfamily{lmss}\selectfont LDPs} for the processes $(L\vee 1)^{\eps}$ as $\eps\downarrow0$, but
these large deviations regimes are very different from the one considered here. 
We will recall how to define a distance function $\dist$ that induces the M1' topology and  
makes $(\D[0,1],\dist)$ a Polish space. It is in this space that a strong {\fontfamily{lmss}\selectfont LDP} holds for
the excursion and bridge of the process $L$.

It might be the case that a weak LDP holds for $\eps\Lex$ in the J1 topology, or for the unconditioned scaled processes $\eps L$ in either the J1 topology or in our modified M1 topology, but we do not pursue these questions here.

 \subsection{Main results}\label{sec:main-results}

 Let $\Lex=(\Lex_t,0\leq t\leq 1)$ be the normalized excursion of $L$ above its past
 infimum 
 (see \cite{Chaumont}), and let $\br^{(x)}=(\br^{(x)}_t,0\leq t\leq t)$ be the bridge of $L$ from $0$
 to $x\in \R$ with unit duration. We set $\br=\br^{(0)}$. 
These processes satisfy a.s.\
$\Lex_0 = \Lex_1 =
 \br^{(x)}_0=0$, $\br^{(x)}_1=x$, and $\Lex_t >0$ for every
 $t\in(0,1)$.
 We view $\Lex$ and $\br^{(x)}$ as random variables in the space $\D[0,1]$ of ``càdlàg'' functions
 $f:[0,1]\to \R$, that is, functions which are right-continuous at
 every point $t\in [0,1)$ and have left-limits at every point $t\in (0,1]$. We denote by $f(t+)=f(t)$ and
 $f(t-)$ the right- and left-limits of $f\in \D[0,1]$ at $t$, whenever applicable.
We turn $\D[0,1]$ into a measurable space by equipping it with the $\sigma$-algebra generated by the
evaluation maps $f\mapsto f(t)$ for $t\in [0,1]$.

 Let us recall some standard definitions from the theory of large deviations. If $S$
 is a topological space, a {\em rate
 function} is a lower semicontinuous function $I:S\to [0,\infty]$,
 \textit{i.e.}\ a function such that the level sets $\LL_I(c)=\{x\in S:I(x)\leq c\}$ are
 closed for every $c\geq 0$.  A rate function is called {\em
   good} if the sets $\LL_I(c),c\geq 0$ are compact. 
 
 \begin{definition}\label{def:LDP}
 Let $\beta>0$ be fixed. A family $(X_\eps)_{\eps>0}$ of random
 elements in the space $S$ (endowed with the completed Borel $\sigma$-algebra) is
 said to satisfy the \textbf{large deviation principle} ({\fontfamily{lmss}\selectfont LDP}) with
 speed $\eps^{-\beta}$, and rate function $I$, if for every Borel set $A\subseteq S$,
 $$ - \inf_{x\in \overset{\circ}{A}} I(x) \leq
 \liminf_{\varepsilon\downarrow0} \eps^{\beta} \log \Prob{ X_\eps \in
   A } \leq \limsup_{\eps\downarrow0} \eps^{\beta} \log \Prob{ X_\eps
   \in A } \leq - \inf_{x \in \bar{A}} I(x). $$ 
 The family $(X_\eps)_{\eps>0}$ is said to be \textbf{exponentially tight} with speed $\eps^{-\beta}$, if for each $M<+\infty$, there exists a compact set $K_M$ such that
 \begin{equation}\label{eq:tight}
 \limsup_{\eps\downarrow0} \eps^{\beta} \log \Prob{ X_\eps \in
   K_M^{\text{c}} } \le - M. 
 \end{equation}
 \end{definition}

We define a rate function in the following
way.
Assume that $f\in \D[0,1]$ has bounded variation, meaning that it can be
written as the difference $g-h$ of two non-decreasing functions
$g,h\in \D[0,1]$. 
This decomposition is not unique; however, it becomes unique if we
further require that $g(0)=0$ and that the Stieltjes measures $\d g$
and $\d h$ are mutually singular. This ``minimal'' decomposition is
classically called the \textit{Jordan decomposition} of $f$, and we
write $g=\fup, h=\fdo$. 
 We denote by $H^{(\alpha)}$ the subspace of $\D[0,1]$ of functions
 $f$ with bounded variation such that $\fdo$ is an absolutely
 continuous function with derivative $\fdo' \in \Lp^{\alpha'}[0,1]$,
 where $\alpha'=\alpha/(\alpha-1)\in (2,\infty)$ is the conjugate exponent of
 $\alpha$. 
Set
\begin{equation}\label{eq:5}
 \Dex[0,1]  = \big\{ f \in \D[0,1]: \: f(1) = 0, \, f \geq 
   0 \big\} , 
\end{equation}
(note that we do not impose the usual condition that $f(0)=0$) and define
\begin{equation}
   \Hex = H^{(\alpha)} \cap \Dex[0,1] .\label{eq:Hex} 
\end{equation}
Let us define the rate function $I_\Lex:\D[0,1]\to [0,\infty]$ by the formula
 \begin{equation}\label{eq:rate-function}
   I_\Lex(f) =\begin{cases}
                        \displaystyle{c_\alpha \int_0^1 (\fdo'(s))^{\alpha'} \d s}
                        &\mbox{ if }  f\in\Hex, \\
                        +\infty  &\mbox{ otherwise}.
                      \end{cases}
      \end{equation}
                    Alternatively, we may
 define $I_\Lex$ as follows for nonnegative functions $f$ with bounded
 variation and such that $f(1)=0$. Write $f=f_{\mathrm{ac}}+f_{\mathrm{sing}}$ as
 a sum of an 
 absolutely continuous part and a singular part, and let
 $f_{\mathrm{sing}}=f_{\mathrm{sing}\uparrow}-f_{\mathrm{sing}\downarrow}$ be the Jordan
 decomposition of $f_{\mathrm{sing}}$. Then we have
 \begin{equation}
   \label{eq:17}
    I_\Lex(f)=c_\alpha\int_0^1(f_{\mathrm{ac}}'(s))_-^{\alpha'}\d s + \infty
    \cdot f_{\mathrm{sing}\downarrow}(1) ,
  \end{equation}
  where $(f_{\mathrm{ac}}'(s))_-$ denotes the negative part of
 $f_{\mathrm{ac}}'(s)$, 
and we let $I_\Lex(f)=\infty$ if $f$ is not nonnegative, or if $f$
does not satisfy $f(1)=0$, or if $f$ does not have bounded
variation. Here, the term $\infty\cdot f_{\mathrm{sing}\downarrow}(1)$
accounts for the fact that \eqref{eq:rate-function} yields $+\infty$
when $f_\downarrow$ is not absolutely continuous. 
In
this way, we note that the shape of the rate function is exactly that
involved in \cite[Theorem 5.1]{Lynch} (although this theorem does not
apply in our context) and the other references mentioned
earlier in this introduction.

We defer a discussion of the topology until Section
\ref{sec:basic-results-skor}, where we will introduce the distance
$\mathrm{dist}$ on the set $\D[0,1]$. 
We can now state our main result. 

  \begin{theorem}[{\fontfamily{lmss}\selectfont LDP} for the normalized excursion $\Lex$]\label{thm:LDP-excursion}
The laws of $(\eps\Lex_t)_{t\in[0,1]}$ satisfy an {\fontfamily{lmss}\selectfont LDP} in
$(\D[0,1],\dist)$ as $\eps\downarrow0$ with speed $\eps^{-\alpha'}$
and good rate function $I_\Lex$. 
 \end{theorem}

We also have the following negative result, proved in Section \ref{subsection:no-J1}.

 \begin{proposition}\label{thm:no-J1}
The rate function $I_\Lex$ is not a good rate function for the Skorokhod J1
topology. Moreover, the laws of $(\eps\Lex_t)_{t\in[0,1]},0<\eps<1$, are not exponentially tight
   in $\D[0,1]$ endowed with the Skorokhod J1 topology. 
 \end{proposition}

 Recall that in a Polish space, an {\fontfamily{lmss}\selectfont LDP} 
 with a good rate function
 implies exponential tightness (see \cite[Lemma 2.6]{Lynch} or
 \cite{Grunwald}). However, since the rate function $I_\Lex$ is not good, this does not rule
 out the possibility that 
 $\eps \Lex$ satisfies the {\fontfamily{lmss}\selectfont LDP} in 
 the J1 topology, and we do not know
 whether this property holds or not. Note that this problem has been
 the topic of several references dealing with random walks and Lévy processes under the
 Cramér condition, including
 \cite{Lynch,mogulskii93,BoMo13}, and at present there
 is no complete answer to this 
 question in that context either. However, we believe that 
 the M1 topology is a more natural choice in this context, since it is
arguably a strong topology for which the
 rate functions are better behaved.

 Theorem \ref{thm:LDP-excursion} allows us to deduce  general 
 {\fontfamily{lmss}\selectfont LDPs}
 for functionals of the normalized stable excursion. This extends 
the results of \cite{Fill-Janson}
 dealing with  Brownian excursions to the case of stable excursions, and was
 the initial motivation for the present work. 
 Define the sets
 \begin{equation}
   \label{eq:6}
     K^{(\alpha)}=\big\{f \in H^{(\alpha)}:\Vert \fdo' \Vert_{\alpha'} 
       \leq 1\big\}\, ,\qquad \Kex  = K^{(\alpha)} \cap \Dex[0,1] .
   \end{equation}
   It will be shown in Lemma \ref{sec:proof-crefthm:ldp-fu} below that $\Kex$ is a compact subset of
   $(\D[0,1],\dist)$. This will imply, using the contraction
 principle, the following logarithmic asymptotics for the right 
 tails of functionals of  $\Lex$. 
 
 \begin{theorem}[Logarithmic asymptotics for the right tails of functionals of $\Lex$]\label{thm:LDP-functional}
 Let $\Phi$ be a continuous nonnegative
 functional $\Dex[0,1]\to \R_+$ which is also positive-homogeneous in the
 sense that $\Phi(\lambda f) = \lambda \Phi(f)$ for every $f \in
 \Dex[0,1]$ and $\lambda \geq 0$, and not identically $0$ on $\Kex$.
Define  $X=\Phi(\Lex)$ and 
let 
 $$ \gamma_\Phi = \max\big\{ \Phi(f): \: f \in \Kex \big\} .$$
 Then $\eps\Phi(\Lex)$ satisfies an {\fontfamily{lmss}\selectfont LDP} in $\R_+$ as $\eps\downarrow0$ with speed $\eps^{-\alpha'}$ and good rate function $J_\Phi(x)=c_\alpha\left(\frac{x}{\gamma_\Phi}\right)^{\alpha'}$. 
 In particular, 
 \begin{eqnarray}
 - \log \Prob{ X>x } \sim
   c_\alpha\left(\frac{x}{\gamma_\Phi}\right)^{\alpha'} & & \quad
                                                            \textrm{
                                                            as }
                                                            x\rightarrow+\infty\, . \label{eq:asymp-functional}
 \end{eqnarray}
\end{theorem}

Using \cite[Theorem 4.5]{Chassaing-Janson}, we have that 
\eqref{eq:asymp-functional} implies the following 
asymptotics for the Laplace transform and the moments: 
 \begin{alignat}{3}
 \log \E{ e^{tX} } & \sim (\gamma_\Phi t)^\alpha && \quad\quad \textrm{ as } t\rightarrow+\infty, \label{eq:asymp-generating-function} \\
 \E{ X^n }^{1/n} & \sim \alpha^{\frac{1}{\alpha}}\gamma_\Phi\left(\frac{n}{ e } \right)^{1/\alpha'} && \quad\quad \textrm{ as } n\rightarrow+\infty. \label{eq:asymp-moments}
 \end{alignat}

 Taking as a particular case the functions $\Phi(f)=\int_0^1f(s) \d s$
 and $\Phi(f)=\sup_{s\in [0,1]}f(s)$, we obtain the following result, 
which improves \cite[Corollary 1.2]{Profeta} by pinning down the
precise constants. 
 
 \begin{corollary}[Logarithmic asymptotics for the right tails of the area under $\Lex$]\label{cor:exact-asymp-Profeta}
Set 
 $$ \Aex = \int_0^1 \Lex_t \, \d t .$$
  Then  it holds that
 \begin{alignat}{3}\label{eq:1}
 -\log \Prob{ \Aex > x }
 & \sim c_\alpha(\alpha+1)^{\frac{1}{\alpha-1}}
x^{\alpha'} && \quad\quad \textrm{ as } x\rightarrow+\infty, \\
 \log \E{ e^{t \Aex} }
 & \sim \frac{ t^\alpha}{\alpha+1} && \quad\quad \textrm{ as } t\rightarrow+\infty, \\
 \E{ \Aex^n }^{1/n} 
 & \sim \left(\frac{\alpha}{\alpha+1}\right)^{\frac{1}{\alpha}} \left(\frac{n}{e}\right)^{1/\alpha'} && \quad\quad \textrm{ as } n\rightarrow+\infty.
 \end{alignat}
 \end{corollary}

\begin{corollary}[Logarithmic asymptotics for the right tails of the supremum of $\Lex$]\label{cor:exact-asymp-Profeta-sup}
 It holds that
 \begin{alignat}{3}\label{eq:18}
 -\log \Prob{ \sup_{0\le t\le1} \Lex_t > x }
 & \sim c_\alpha x^{\alpha'} && \quad\quad \textrm{ as } x\rightarrow+\infty, \\
 \log \E{ e^{t \sup_{0\le s\le1} \Lex_s} }
 & \sim t^\alpha && \quad\quad \textrm{ as } t\rightarrow+\infty, \\
 \E{ \Big(\sup_{0\le t\le1} \Lex_t\Big)^n }^{1/n} 
 & \sim \alpha^{1/\alpha} \left(\frac{n}{e}\right)^{1/\alpha'} && \quad\quad \textrm{ as } n\rightarrow+\infty.
 \end{alignat}
 \end{corollary}

\subsection{Large deviation principles for bridges}\label{sec:large-devi-princ}

Theorems \ref{thm:LDP-excursion} and \ref{thm:LDP-functional} have
counterparts for bridges of the Lévy process $L$. For $a\in \R$, we
let 
 \begin{align*}
 \Dbra[0,1] & = \left\{ f \in \D[0,1]: \: f(1) = a \right\} , \\
 \Hbra & = H^{(\alpha)} \cap \Dbra[0,1] . 
 \end{align*}
 We may now state the main results concerning the stable L\'evy bridge.
 In this statement and the rest of the paper, for $a\in\R$, we let
 $a_+=a\vee 0$ and $a_-=(-a)_+$ be the positive and negative parts of
 $a$. 
  
 \begin{theorem}[{\fontfamily{lmss}\selectfont LDP} for the stable bridge $\bra$]\label{thm:LDP-bridge}
 Let $(a_\eps)_{\eps>0}$ be such that $\eps a_\eps \to a$ as
 $\eps\to0$. Then the laws of $(\eps\brae_t)_{t\in[0,1]}$ satisfy an
 {\fontfamily{lmss}\selectfont LDP} in $(\D[0,1],\dist)$ as
 $\eps\downarrow0$ with speed $\eps^{-\alpha'}$ and good rate function
 $I_{\br,a}$ defined by 
 \begin{equation}
 I_{\br,a}(f) = \begin{cases} c_\alpha \left(\int_0^1
     |\fdo'(s)|^{\alpha'} \d s - (a_-)^{\alpha'}\right) &

   \mbox{if } f\in\Hbra, \\+\infty & \mbox{ otherwise. }
 \end{cases}
\end{equation}
\end{theorem}

We obtain an analogue of Theorem \ref{thm:LDP-functional} for bridges. Let
$$ K_{\mathrm{br}} = K^{(\alpha)} \cap D_{\mathrm{br}}^{(0)}[0,1] . $$
 Again, $K_{\mathrm{br}}$ is a compact subset of
 $(\D[0,1],\mathrm{dist})$, and the following logarithmic asymptotics hold for the right 
 tails of functionals of  $\br$. 
 
 \begin{theorem}[Logarithmic asymptotics for the right tails of functionals of $\br$]\label{thm:LDP-functional-bridge}
 Let $\Phi$ be  a continuous
 nonnegative functional $\Dbr^{(0)}[0,1]\to\R_+$ which is also
 positive-homogeneous in the sense that $\Phi(\lambda f) = \lambda
 \Phi(f)$ for every $f \in \Dbr^{(0)}[0,1]$ and $\lambda \geq 0$, and not
 identically $0$ on $K_{\mathrm{br}}$. Define $X = \Phi(\br)$ and let 
 $$ \gamma_\Phi = \max\left\{ \Phi(f): \: f \in K_{\mathrm{br}} \right\} . $$
 Then $\eps\Phi(\br)$ satisfies an {\fontfamily{lmss}\selectfont LDP}
 in $\R_+$ as $\eps\downarrow0$ with speed $\eps^{-\alpha'}$ and good rate function $J^\br_\Phi(x)= c_\alpha\left(\frac{x}{\gamma_\Phi}\right)^{\alpha'}$.
 In particular, 
 \begin{eqnarray*}
 - \log \Prob{ X>x } \sim c_\alpha\left(\frac{x}{\gamma_\Phi}\right)^{\alpha'} & & \quad \textrm{ as } x\rightarrow+\infty . 
 \end{eqnarray*}
 \end{theorem}

 As an application, using the same proof as for Corollary \ref{cor:exact-asymp-Profeta-sup}, we may
 reprove an exact logarithmic asymptotic for the right tails of the
 supremum of the stable L\'evy bridge obtained by 
 \cite{Kortchemski}. 
 
 \begin{corollary}[{\cite[Corollary 13]{Kortchemski}}]\label{cor:Kor}
 We have
  $$ -\log \Prob{\sup_{0\le t\le 1}\br_t>x} \sim c_\alpha x^{\alpha'} . $$
 \end{corollary}

 \subsection{Outline of the proofs and organization of the paper}
 
The proofs for excursions and
 bridges are very much alike, but some extra technicalities arise for
 excursions, so we focus mostly on that case, and deal with bridges in Section
 \ref{section:bridge}. 
 Section \ref{section:preliminaries} will recall the basics of stable
 processes, bridges and excursions, as well as the results on the M1
 topology that will be needed in this paper.

 In order to prove Theorem \ref{thm:LDP-excursion}, we first establish
    the {\fontfamily{lmss}\selectfont LDP} for the finite-dimensional
    marginals of $\eps \Lex$ (Proposition \ref{prop:LDP-finite-dim}). This
     relies on the explicit form of the finite-dimensional
    marginals of the stable excursion in terms of stable densities and
    related quantities, which is recalled in Section \ref{sec:excurs-bridg-spectr}. The key input is the
    following estimate for stable densities $p_t(x)=\Prob{L_t\in \d
    x}/\d x$ 
  (see \cite[Equation (14.35)]{Sato}),
  \begin{equation}
    \label{eq:19}
       \begin{cases} p_1(x) = C_\alpha x^{-\alpha-1} \big(1+
      O(x^{-\alpha})\big) & \text{ as } x\to+\infty \\ p_1(-x) =
    c_\alpha''  x^{\frac{2-\alpha}{2\alpha-2}} \exp\big(-c_\alpha
      x^{\alpha'}\big) \big(1+O(x^{-\alpha'})\big) & \text{ as }
    x\to +\infty . \end{cases}
\end{equation}
The asymmetry of these two asymptotic behaviors will play a key role. This will imply that $\eps\Lex$ satisfies
a large deviation principle for the weak topology on $\D[0,1]$ of
pointwise convergence at continuity points of the limit. This result
is proved in Section \ref{section:thm1}, which is also devoted to the
identification of the rate function. In order to
prove an {\fontfamily{lmss}\selectfont LDP} in $(\D[0,1],\mathrm{dist})$, we show that the laws of
$(\eps\Lex,\eps\in (0,1))$ are exponentially tight in this space. This
relies on a Kolmogorov-type criterion which we prove in Section
\ref{sec:an-expon-tightn}, and apply to our present context in Section
\ref{section:expo-tight}. Finally, Theorem \ref{thm:LDP-functional} is
proved in Section \ref{section:thm2} using the ideas of 
\cite{Fill-Janson}, who treated the Brownian case.

\section{Preliminaries}\label{section:preliminaries}

 \subsection{Excursions and bridges of stable Lévy processes
   without negative jumps}\label{sec:excurs-bridg-spectr}

 We begin by recalling the definitions of stable excursions and bridges, for
 which we mainly refer to \cite{Chaumont}. We denote by $\bP_x$ the
 law under which the canonical càdlàg process $(L_t,t\geq
 0)$ is a stable L\'evy process without negative jumps with exponent $\alpha$,
 started from $x$,  and we set $\bP=\bP_0$. We
 let $(\Fil_t)_{t\geq0}$ be the natural filtration. We
 denote by $(p_t)_{t\geq0}$ the continuous transition semigroup
 density of $L$ under $\bP_x$, which possesses the scaling property 
 $$ p_t(x) = t^{-1/\alpha} p_1(t^{-1/\alpha}x) . $$
 Let $\mathcal{E}$ be the excursion space, which is defined by
 $$ \mathcal{E} = \big\{ \omega \in \D(\R_+,\R_+): \, \omega(0)=0 \textrm{ and } \zeta(\omega) = \sup\{ t>0, \: \omega(t)>0 \} \in (0,\infty) \big\} . $$
 Denote by $\underbar{n}$ the It\^o measure of $L$ above its past
 infimum.
 The $\sigma$-finite ``law''  of the lifetime under $\underbar{n}$ has
 been calculated by 
 \cite[Lemma 3.1]{Monrad-Silverstein}:
 $$ \underbar{n}( t<\zeta ) = \Gamma\left(1-\frac{1}{\alpha}\right)^{-1} t^{-1/\alpha} . $$
 For $\lambda>0$, we define the scaling operator $S^{(\lambda)}$ on $\mathcal{E}$ by
 $$ S^{(\lambda)}(\omega) = ( \lambda^{1/\alpha} \omega_{t/\lambda}, \, t\geq0 ) . $$
 Then there exists a unique collection of probability measures $(\underbar{n}^{(t)}, \, t>0)$ on $\mathcal{E}$ such that
 \begin{enumerate}[label=(\roman*)]
 \item for every $t>0$, $\underbar{n}^{(t)}(\zeta=t)=1$;
 \item for every $\lambda>0$ and $t>0$, we have $S^{(\lambda)}(\underbar{n}^{(t)}) = \underbar{n}^{(\lambda t)}$;
 \item for every measurable subset $A$ of $\mathcal{E}$,
 $$ \underbar{n}(A) = \int_0^\infty \frac{\d s}{\alpha \Gamma\left(1-\frac{1}{\alpha}\right) s^{1+\frac{1}{\alpha}}} \underbar{n}^{(s)}(A) . $$
 \end{enumerate}
 The probability distribution $\underbar{n}^{(1)}$ on c\`adl\`ag paths
 with unit lifetime is called the law of the normalized excursion of
 $L$.
 
 We denote by $\bP_x^{(0,\infty)}$ the law of the process $(L,\bP_x)$,
 $x>0$, killed when it leaves $[0,\infty)$, so that 
 $$ \bP_x^{(0,\infty)}\left( A, \, t<\zeta \right) =
 \Prob{ A, \, t<\tau_{(-\infty,0)} }, \quad t\geq0, \quad
 A \in \Fil_t . $$
 We denote by $(p_t^{(0,\infty)}(x,\cdot))_{t\geq0}$ the transition
 semigroup under $\bP_x^{(0,\infty)}$. 
 The measure $\underbar{n}$ is Markovian with semigroup
 $(p_t^{(0,\infty)}(x,\cdot))_{t\geq0}$ under $\bP_x$, which means
 that if $F$ is measurable and nonnegative, and $\theta_tf=f(t+\cdot)$ is the
 shift operator, then
 \begin{equation}
   \label{eq:7}
     \underbar{n}\big( \mathbbm{1}_A\, F \circ \theta_t \,
     \1{t<\zeta} \big) = \underbar{n}\big( \mathbbm{1}_A \,
     \bE_{L_t}^{(0,\infty)}\left[F\right] \, \1{t<\zeta} \big), \quad
     t\geq0, \quad A \in \Fil_t .
   \end{equation}
   We denote by $(q_x(t))_{t\geq0}$ the density of the first hitting
   time of $0$ under $\bP_x^{(0,\infty)}$. Thanks to the absence of
   negative jumps, the density $(q_x(t))_{t\geq0}$ can be related to
   the law of $L$ as follows (see, for instance,
   \cite[Corollary VII.3]{Bertoin}):
   $$ q_x(t) = \frac{x}{t} p_{t}(-x) . $$
 Hence, it satisfies the following scaling property
 $$ q_x(t) = x^{-\alpha} q_1(x^{-\alpha} t) . $$
 Let $(j_t)_{t\geq0}$ be the density of the entrance law under the
 measure 
 $\underbar{n}$, defined by the fact that, for every $t>0$,
 $$ \underbar{n}( f(L_t) \1{t<\zeta} ) = \int_0^\infty f(x) j_t(x) \d x , $$
 where $f$ is an arbitrary bounded Borel function.
 Recall that for all $t>0$, $j_t$ is an integrable function in
 $\Lp^\infty(\R_+)$, and we may choose it so that it satisfies the following scaling property 
 \begin{equation}\label{eq:3}
 j_t(x) = t^{-2/\alpha} j_1(t^{-1/\alpha}x) 
 \end{equation}
(see \cite[Lemma 3.2]{Monrad-Silverstein}). 
 Combining the above expressions, we may show that the law of the
 normalized excursion has a density with respect to the Lebesgue
 measure. Indeed, if $f:\R \rightarrow\R_+^*$ and
 $g:\R_+\rightarrow\R_+$ are two nonnegative measurable functions, we then have
 \begin{eqnarray*}
 \underbar{n}\left( f(L_t) g(\zeta) \1{t<\zeta} \right)
 & = & \int_t^\infty \d s g(s) \int_\R f(x) j_t(x) q_x(s-t) \d x \\
 & = & \int_t^\infty \frac{g(s)}{\alpha \Gamma\left(1-\frac{1}{\alpha}\right) s^{1+\frac{1}{\alpha}}} \underbar{n}^{(s)}( f(L_t) ) \d s .
 \end{eqnarray*}
 This implies that for all $s>0$,
 $$ \underbar{n}^{(s)}( f(L_t) ) = \alpha \Gamma\left(1-\frac{1}{\alpha}\right) \int_{\R_+} f(x) j_t(x) q_x(s-t) \d x . $$
 In particular, when $s=1$ the law of the normalized excursion is then
 $$ \underbar{n}^{(1)}( f(L_t) ) = \alpha \Gamma\left(1-\frac{1}{\alpha}\right) \int_{\R_+} f(x) j_t(x) q_x(1-t) \d x . $$
 Using a similar argument and the Markov property, we can compute
 \begin{equation}\label{eq:marginal-laws}
 \underbar{n}^{(1)}( f(L_{t_1},\dots,L_{t_n}) )= \alpha \Gamma\left(1-\frac{1}{\alpha}\right) \int_{\R_+^n} j_{t_1}(x_1) \prod_{i=1}^{n-1} p^{(0,\infty)}_{t_{i+1}-t_i}(x_i,x_{i+1}) q_{x_n}(1-t_n) \d x_1 \,\cdots \d x_n .
 \end{equation}
This gives the finite-dimensional marginals for the excursion
process $\Lex$, so that the left hand-side of the preceding equation
may also be written as $\E{f(\Lex_{t_1},\ldots,\Lex_{t_n})}$. In particular, for every $t\in (0,1)$, the law
of $(\Lex_{t+s},0\leq s\leq 1-t)$ can be obtained from the following
formula, valid for every non-negative measurable $F$: 
\begin{equation}
  \label{eq:15}
  \E{F(\Lex_{t+s},0\leq s\leq 1-t)}=\alpha\Gamma\left(1-\frac{1}{\alpha}\right)\int_{\R_+}j_t(x)q_x(1-t)\bE^{1-t}_x[F(L)] \d x
  ,
\end{equation}
where $\bE^\delta_x$ is the law of the first-passage bridge of duration
$\delta>0$ started from $x>0$. The latter is defined by absolute continuity
for every $\delta'>0$ and $\Fil_{\delta'}$-measurable $F\geq 0$ by the
formula
\begin{equation}
  \label{eq:16}
\bE^\delta_x[F]=\bE_x\left[F
  \1{T_0>\delta'}\frac{q_{L_{\delta'}}(\delta-\delta')}{q_x(\delta)}\right] .
\end{equation}

In a similar but simpler way, the law of the bridge $\bra$ has
finite-dimensional marginals given by
\begin{equation}
  \label{eq:8}
  \E{f(\bra_{t_1},\ldots,\bra_{t_n})}=\frac{1}{p_1(a)}\int_{\R^n} f(x_1,\ldots,x_n)\prod_{i=0}^{n}
p_{t_{i+1}-t_i}(x_{i+1}-x_i) \d x_1 \,\cdots \d x_n ,
\end{equation}
where $0=t_0<t_1<\ldots<t_n<1=t_{n+1}$, and by convention we let
$x_0=0$ and $x_{n+1}=a$ in the above integral.

 \subsection{Basic results on the Skorokhod M1' topology}\label{sec:basic-results-skor}
 
 In this section, we introduce the distance $\dist$ on
 $\D[0,1]$ inducing the Skorokhod M1' topology considered in \cite{BBRZ20,V21}, and study some of its key properties.
  We also refer to \cite[Chapters 12 and 13]{Whitt} for a complete
  treatment of the Skorokhod M1 topology.

 By convention, for $f\in \D[0,1]$, we let $f(0-)=0$,  which is a way
 to ``root'' the function $f$ at $0$: by contrast, we adopt the convention
 $f(1+)=f(1)$. 

For two real numbers $x,y\in \R$, the real interval $[x\wedge y,x\vee y]$ will more simply be denoted by $[x,y]$, even when $x>y$.

 \subsubsection{The space $(\D[0,1],\dist)$}

 For $f\in \D[0,1]$, we define the {\em augmented graph} of $f$ {\em
   rooted at $0$}, to be the set
 $$ \Gamma_0(f) = \big\{ (t,x) : t\in [0,1], x\in [f(t-),f(t)] \big\}\subset [0,1]\times \R .$$
 Note in particular that $\Gamma_0(f)$ contains the segments $\{0\}\times [0,f(0)]$ and $\{1\}\times [f(1-),f(1)]$.
 For $(t,x),(u,y)\in \Gamma_0(f)$, we write $(t,x)\preceq(u,y)$ if $t<u$
 or if $t=u$ and $|x-f(t-)|\leq |y-f(t-)|$. This defines a total order
 on $\Gamma_0(f)$. We say that a function $r\in [0,1]\mapsto
 (t(r),x(r))$ is a {\em parametric representation} of $\Gamma_0(f)$ if
 it is an increasing bijection from $([0,1],\leq)$ to
 $(\Gamma_0(f),\preceq)$, and we write $\Pi(f)$ for the set of all
 parametric representations of $\Gamma_0(f)$.
 For $f_1,f_2\in \D[0,1]$, we let 
 $$ \dist(f_1,f_2) = \inf \left\{ \sup_{r\in[0,1]} |t_1(r)-t_2(r)|\vee|x_1(r)-x_2(r)|:(t_1,x_1)\in \Pi(f_1),(t_2,x_2)\in \Pi(f_2) \right\} .$$
 This indeed defines a distance function that makes $(\D[0,1],\dist)$
 a Polish space\footnote{Note, however, that the distance $\dist$ is
   not complete; see \cite[Section 12.8]{Whitt}.}. These results are
 obtained exactly as in the context of the M1 topology (see \cite[Theorems 12.3.1 and 12.8.1]{Whitt}), with the slight
 modification that the convention taken in this reference is that $f(0-)=f(0)$ rather than $f(0-)=0$, and that $f$
 is supposed to be continuous at $0$ and $1$. 
 The choice of convention that $f(0-)=0$ allows a sequence
 of functions that jump ``right after time 0'' to be possibly
 convergent in the
M1' topology. 
 For example, one has 
 $\dist(\mathbbm{1}_{[1/n,1]},\mathbbm{1}_{[0,1]})\to 0$ as
 $n\to\infty$, while the sequence $\mathbbm{1}_{[1/n,1]}$ is not
 convergent in the classical M1 topology, and in fact 
the 
 M1' topology is
 strictly weaker than the M1 topology. 
 On the other hand, observe that $\mathbbm{1}_{[0,1/n]}$ is not convergent in $(\D[0,1],\dist)$.

\subsubsection{The $M$-oscillation}

In this section, we introduce an oscillation function that will serve
as a substitute in $(\D[0,1],\dist)$ for the classical modulus of continuity. 
 For $x\in \R$ and $A\subset \R$ we let $d(x,A)=\inf\{|x-y|:y\in A\}$
 be the distance from $x$ to $A$. Note the elementary inequalities
 \begin{equation}\label{eq:4}
 d(x,A)-d(y,A)\leq d(x,y)\, ,\qquad d(x,A)-d(x,B)\leq d_H(A,B)
 \end{equation}
 where $x\in \R$ and $A,B\subset \R$, and $d_H(A,B)=\sup_{x\in A}d(x,B)\vee \sup_{y\in B}d(y,A)$ is the Hausdorff distance between $A$ and $B$. 

 For $x,y,z\in \R$, we set
 $$ M(x,y,z)=d(y,[x,z])=(y-x)_+\wedge (y-z)_+ + (y-x)_-\wedge (y-z)_-  . $$
 The $M$-oscillation of a function $f\in \D[0,1]$ is the function defined for $\delta>0$ by 
 $$ w_M(f,\delta)=\sup\big\{M(f(t_1-),f(t),f(t_2)):0\leq t_1< t< t_2\leq 1, |t_2-t_1|< \delta\big\} . $$
 The choice of the left-limit at $t_1$ in the first term might appear unnatural at first sight, because of the fact that $f$ has left limits. In fact, it is only needed when $t_1=0$, because of our rooting convention $f(0-)=0$. So in fact, $w_M(f,\delta)$ is the maximum of the two quantities
 $$\sup\big\{M(f(t_1),f(t),f(t_2)):0\leq t_1< t<t_2\leq 1, |t_2-t_1|< \delta\big\}$$
 and
 $$\sup\big\{M(0,f(t),f(t_2)):0< t< t_2< \delta\big\} ,$$
 and if $f(0)=0$, then $w_M(f,\delta)$ is equal to the first quantity. 
 
 \begin{theorem}\label{sec:basic-results-skor-1}
 Let $\mathcal{D}$ be a fixed countable subset of $[0,1]$ containing $1$. 
 Let $K\subset \D[0,1]$ be such that
 $$ \sup\big\{|f(q)|:f\in K,q\in \mathcal{D}\big\}<\infty $$
 and 
 $$ \lim_{\delta\downarrow 0}\sup\big\{w_M(f,\delta):f\in K\big\} =0 . $$
 Then $K$ is a relatively compact subset of $(\D[0,1],\dist)$. 
 \end{theorem}
 
 This theorem can be found in Chapter 12 of 
 \cite{Whitt}, with
 the minor difference that our space of functions starts with an ``initial jump''  (recall that $f(0-)=0$ by our convention).

\subsubsection{An exponential tightness criterion}\label{sec:an-expon-tightn}

Our second result gives a sufficient condition to check exponential tightness in $(\D[0,1],\dist)$ for a family of random processes.

 \begin{theorem}[Exponential tightness in $(\mathbb{D}{[}0,1{]},\dist)$]\label{sec:basic-results-skor-2}
 Let $\alpha>1$ and $\{ X_{(\eps)},\eps >0 \}$ be a family of
 $\D[0,1]$-valued stochastic processes. We assume that
 $X_{(\eps)}(0)=0$ for every $\eps>0$.
 \begin{enumerate}[label=\arabic*.]
 \item\label{item tight 1} Suppose that 
 there exist two constants $c,C\in (0,\infty)$ such that for every $\lambda>0$, and every $0\leq t_1\leq t\leq t_2\leq 1$, 
 $$ \E{\exp(\lambda M(X_{(\eps)}(t_1),X_{(\eps)}(t),X_{(\eps)}(t_2)))}
 \leq c\exp\big(C(\lambda\,  \eps)^\alpha|t_2-t_1|\big) , $$
 then it holds that for every $\beta\in (0,1/\alpha)$, 
 $$\lim_{N\to\infty}\limsup_{\eps\downarrow 0}\eps^{\alpha'}\log
 \Prob{\bigcup_{n> N} \left\{w_M(X_{(\eps)},2^{-n})>
   2\frac{2^{-n\beta}}{1-2^{-\beta}}\right\}}=-\infty .$$
 \item\label{item tight 2} 
Suppose further that for some countable dense set $\mathcal{D}$ of
$[0,1]$ containing $1$, the family of random variables $\{
X_{(\eps)}(q),\eps>0 \}$ is exponentially tight with speed
$\eps^{-\alpha'}$, for every $q\in \mathcal{D}$. 
 Then the laws of $X_{(\eps)}$ as $\eps\downarrow 0$ are exponentially
 tight in $(\D[0,1],\dist)$, with speed $\eps^{-\alpha'}$.
\end{enumerate}
\end{theorem}

\begin{proof}
 For \ref{item tight 1}, 
 we follow and adapt the approach of 
 \cite{Bilingsley}. 
 Let $D_n=\{k2^{-n},0\leq k\leq 2^n\}$ be the dyadic numbers of level
 $n$. Then for $\beta>0$ and $\lambda>0$,  by Markov's inequality, 
 $$ \Prob{M(X_{(\eps)}(k2^{-n}),X_{(\eps)}((k+1)2^{-n}),X_{(\eps)}((k+2)2^{-n})>2^{-n\beta}}
 \leq c\exp(-\lambda 2^{-n\beta}+C(\lambda \eps)^\alpha2^{-n+1}) $$
 which, by optimizing over $\lambda>0$ and taking a union bound, yields
 \begin{align*}
 \Prob{\exists k\in\{0,1,\ldots,2^n-2\}:M(X_{(\eps)}(k2^{-n}),X_{(\eps)}((k+1)2^{-n}),X_{(\eps)}((k+2)2^{-n}))>2^{-n\beta}}\\
 \leq
   2^nc\exp(-C'(\alpha)\eps^{-\alpha'}2^{n(1-\alpha\beta)/(\alpha-1)})\, .
 \end{align*}
 Setting $A_n=\max\{M(f(k2^{-n},(k+1)2^{-n},(k+2)2^{-n})),0\leq k\leq
 2^n-2\}$, this shows that 
 $$
 \Prob{A_n>2^{-n\beta}}\leq
 2^nc\exp(-C'(\alpha)\eps^{-\alpha'}2^{n(1-\alpha\beta)/(\alpha-1)}) . $$ 
 Next, let $f\in \D[0,1]$ and, for $I\subset [0,1]$, let
 $$\LL(I)=\sup\big\{M(f(t_1),f(t),f(t_2)):t_1,t,t_2\in I, t_1\leq t\leq t_2 \big\} .$$ 
 Fix $n\geq 1$ and $k\in \{0,1,\ldots,2^n-2\}$. We aim to provide
 bounds on $\LL([k2^{-n},(k+2)2^{-n}])$. To this end, by
 right-continuity, it suffices to bound uniformly the quantities
 $M(f(t_1),f(t),f(t_2))$ for $t_1\leq t\leq t_2$ in
 $[k2^{-n},(k+2)2^{-n}]\cap \bigcup_{m\geq 0}D_m$. For $m\geq n$, let 
 $$ B_m=\max\big\{M(f(t_1),f(t),f(t_2)):k2^{-n}\leq t_1\leq t\leq t_2\leq (k+2)2^{-n}, t_1,t,t_2\in D_m\big\} , $$
 so that $\LL([k2^{-n},(k+2)2^{-n}])$ is the increasing limit of $B_m$ as $m\to\infty$. 
 The key observation is that, for every $m\geq  n$, 
 \begin{equation}\label{eq:2}
 B_m\leq B_{m-1}+2A_m\, .
 \end{equation}
 To check  this, let us assume that $t_1<t<t_2$ are in $D_m$ and
 achieve the maximum defining $B_m$. If $f(t)$ lies between $f(t_1)$ and
 $f(t_2)$ then this means that $B_m=0$ and there is nothing to prove.
 Otherwise, we may assume without loss of generality that $f(t_2)\leq
 f(t_1)<f(t)$, the other cases being symmetric, so that
 $B_m=f(t)-f(t_1)$. Note that if $t\in D_m\setminus D_{m-1}$, then
 $t-2^{-m},t+2^{-m}$ belong to $D_{m-1}$.
Moreover, it must hold that $f(t-2^{-m})\vee f(t+2^{-m})\leq f(t)$, as
otherwise, for instance if $f(t-2^{-m})>f(t)$, then we would have
$$M(f(t_1),f(t-2^{-m}),f(t_2))=f(t-2^{-m})-f(t_1)>f(t)-f(t_1)=M(f(t_1),f(t),f(t_2)) ,$$
and this would contradict the assumption that $M(f(t_1),f(t),f(t_2))$ is
maximal over points $t_1<t<t_2$ in $D_m$. This implies that
$M(f(t-2^{-m}),f(t),f(t+2^{-m}))=(f(t)-f(t-2^{-m}))\wedge
(f(t)-f(t+2^{-m}))$, and therefore, we may choose $t'\in
\{t-2^{-m},t+2^{-m}\}$ such that $|f(t)-f(t')|\leq A_m$. If $t\in
D_{m-1}$, we let $t'=t$.

We define $t'_1$ in a similar way, setting it to be $t_1$ if the latter
belongs to $D_{m-1}$. If $t_1\in D_m\setminus D_{m-1}$, on the other hand, then $t_1\pm
2^{-m}$ belong to $D_{m-1}$. We note that $f(t_1- 2^{-m})\wedge f(t_1+2^{-m})\geq
f(t_1)$, as otherwise this would again contradict the maximality of
$M(f(t_1),f(t),f(t_2))$ over points in $D_m$. So we may choose
$t'_1\in \{t\pm2^{-m}\}$ in such a way that $|f(t'_1)-f(t_1)|\leq
A_m$.

Finally, we define $t'_2$ in the following way. If $t_2\in D_{m-1}$, we
let $t'_2=t_2$ as usual. If $t_2\in D_{m}\setminus D_{m-1}$, we have
two situations. If $f(t_2-2^{-m})\wedge f(t_2+2^{-m})\geq f(t_2)$ then
we may again choose $t'_2\in \{t_2\pm 2^{-m}\}$ in such a way that
$|f(t'_2)-f(t_2)|\leq
A_m$. In this case, we notice that
$d_H([f(t_1),f(t_2)],[f(t'_1),f(t'_2)])\leq A_m$, so that 
\begin{equation}
  \label{eq:14}
  d(f(t),[f(t'_1),f(t'_2)])\geq
  d(f(t),[f(t_1),f(t_2)])-A_m
\end{equation}
by \eqref{eq:4}. 
Otherwise, we choose $t'_2\in \{t_2-2^{-m},t_2+2^{-m}\}$ in such a way
that $f(t'_2)\leq f(t_2)$. In this case, we have
$d_H([f(t_1),f(t'_2)],[f(t'_1),f(t'_2)])\leq A_m$ so that, again by
\eqref{eq:4}, 
$$d(f(t),[f(t'_1),f(t'_2)])\geq
d(f(t),[f(t_1),f(t'_2)])-A_m=d(f(t),[f(t_1),f(t_2)])-A_m ,$$
so that \eqref{eq:14} holds in every case.  
Therefore, for this choice of $t'_1,t',t'_2$ and by \eqref{eq:4}, we obtain
 \begin{align*}
 B_{m-1}\geq M(f(t_1'),f(t'),f(t_2'))&=d(f(t'),[f(t'_1),f(t'_2)])\\
 & \geq d(f(t),[f(t'_1),f(t'_2)]) - A_m\\
 & \geq d(f(t),[f(t_1),f(t_2)])-2A_m=B_m-2A_m\, ,
 \end{align*}
 so that \eqref{eq:2} holds. 
 By taking a limit, \eqref{eq:2} implies that $\LL([k2^{-n},(k+2)2^{-n}])\leq B_{n-1}+2\sum_{m\geq n}A_m$ where we note that $B_{n-1}=0$. Furthermore, we note that 
 $$ w_M(f,2^{-n})\leq \max_{0\leq k\leq 2^n-2}\LL([k2^{-n},(k+2)2^{-n}]) $$
 because three numbers within distance $2^{-n}$ can all be fitted into the same interval $[k2^{-n},(k+2)2^{-n}]$ for some $k$. We deduce that
 $$ w_M(f,2^{-n})\leq 2\sum_{m\geq n}A_m . $$
 Finally, let
 $$ K_N=\bigcap_{n\geq N}\left\{f\in \D[0,1]: w_M(f,2^{-n})\leq 2\frac{2^{-n\beta}}{1-2^{-\beta}}\right\} , $$
 so that 
 \begin{align*}
 \Prob{X_{(\eps)}\notin K_N} & \leq \sum_{n\geq N} \bP\Big(2\sum_{m\geq n}A_m\geq 2\frac{2^{-n\beta}}{1-2^{-\beta}}\Big) \\
 & \leq \sum_{n\geq N} \sum_{m\geq n}\Prob{A_m\geq 2^{-m\beta}} \\
 & \leq \sum_{n\geq N}\sum_{m\geq n}c2^m \exp(-C'(\alpha)\eps^{-\alpha'} 2^{m(1-\alpha\beta)/(\alpha-1)}) \\
 & \leq
   C''2^N\exp(-C'(\alpha)\eps^{-\alpha'}2^{N(1-\alpha\beta)/(\alpha-1)})\, ,
 \end{align*}
 for some universal constant $C''=C''(\alpha)>0$.
 We finally deduce that 
 $$ \limsup_{\eps\downarrow 0} \eps^{\alpha'} \log \Prob{X_{(\eps)}\notin K_N} \leq -C'2^{N(1-\alpha\beta)/(\alpha-1)} , $$
 which converges to $-\infty$ as $N\to\infty$.

 It remains to prove \ref{item tight 2} 
 Notice that for every choice of
 $0=t_0<t_1<\ldots<t_k=1$ in $\mathcal{D}$ with
 $\max\{t_i-t_{i-1}:1\leq i\leq k\}<\delta$, it holds that
$$\sup_{t\in[0,1]}|X_{(\eps)}(t)|\leq \max\big\{|X_{(\eps)}(t_i)|:1\leq
i\leq k\big\}+w_M(X_{(\eps)},\delta) ,$$
so that $$\Prob{\sup_{t\in[0,1]}|X_{(\eps)}(t)|> A}\leq
\sum_{i=1}^k\Prob{|X_{(\eps)}(t_i)|>
  A/2}+\Prob{w_M(X_{(\eps)},\delta)>A/2} .$$
From the fact that the
$X_{(\eps)}(t_i),1\leq i\leq k$ are exponentially tight, and by \ref{item tight 1}, we
obtain the existence of $A_N\in (0,\infty)$ such that 
 $$  \limsup_{\eps\downarrow
   0}\eps^{\alpha'}\log \Prob{\sup_{t\in \mathcal{
     D}}|X_{(\eps)}(t)|> A_N} < -N .$$
 We deduce that the relatively compact sets of $\D[0,1]$ given by $\{f\in \D[0,1]:\sup_{t\in \mathcal{D}}|f(t)|\leq A_N\}\cap K_N$ fulfill the definition of exponential tightness.
 \end{proof}

\section{Large deviations for the finite-dimensional marginal distributions}\label{section:LDP-finite}

 In this section we prove the following proposition.
 
 \begin{proposition}[{\fontfamily{lmss}\selectfont LDP} for the marginals of $\Lex$]\label{prop:LDP-finite-dim}
 Let $\sigma=(t_1,\ldots,t_n)$, where $0<t_1<\dots<t_n<1$, be fixed. Under $\bP$
 the laws of $\eps(\Lex_{t_1},\dots,\Lex_{t_n})$ satisfy an
 {\fontfamily{lmss}\selectfont LDP} in $\R^n$ with speed
 $\eps^{-\alpha'}$ and good rate function 
 $$ J_\sigma(x_1,\dots,x_n) =
 \begin{cases}
   \displaystyle{c_\alpha \sum_{i=1}^n (t_{i+1}-t_i
   )\left(\frac{(x_i-x_{i+1})_+}{t_{i+1}-t_i}\right)^{\alpha'}} & \mbox{
     if }x_1,\ldots,x_n\in \R_+\\
   \infty& \mbox{ otherwise}
 \end{cases}
   $$
   with the convention that $x_{n+1}=0$ and $t_{n+1}=1$.
 \end{proposition}

 The fact that $J_\sigma$ is a good rate function on $\R^n$ is
 easy to see. Indeed, it is clearly continuous, and for every $c>0$, $J(x_1,\ldots,x_n)\leq c$ implies
 $(x_i-x_{i+1})_+\leq c'$ for $1\leq i\leq n$ and $x_n\leq c'$, where $c'$ is some
 positive number depending only on $t_1,\ldots,t_n,c$ 
 that $x_i\leq x_{i+1}+c'$ 
 that $\max_{1\leq i\leq n}x_i\leq n c'$, and the level sets of $J_\sigma$ are
 therefore compact.

 \subsection{Estimates for transition densities}
 
 We will need some crucial estimates for the tails of the transition
 densities, $p_t(x)$, and for the density of the entrance law, $j_t(x)$.
In this section, we will make use of
positive, finite universal constants $c_1,c_2$ depending only on
$\alpha$, but whose values may vary from
line to line, and of non-universal constants $c,C$ depending on some
extra parameters that will always be specified. 
 
First, \cite[Equation (14.35)]{Sato} entails that 
for every $x\geq 0$, we have 
\begin{equation}\label{eq:9}
 c_1 \exp\left(-c_\alpha x^{\alpha'}\right) 
 \leq p_1(-x) \leq 
 c_2(1+ x^{\frac{2-\alpha}{2\alpha-2}}
 )\exp\left(-c_\alpha x^{\alpha'}\right),
\end{equation}
and \cite[Equation (14.34)]{Sato} entails that 
 \begin{equation} \label{eq:10}
 c_1(1+ x)^{-\alpha-1} \leq p_1(x) \leq c_2(1+ x)^{-\alpha-1} .
 \end{equation}

By the scaling relations  for $p_t(x)$, we deduce that for every $x>0$
and $t\in (0,1]$,
\begin{equation}
  \label{eq:11}
   c_1
   \exp\left(-c_\alpha\left(\frac{x}{
         t^{1/\alpha}}\right)^{\alpha'}\right)  
 \leq p_t(-x) \leq 
 \frac{c_2}{t^{1/\alpha}}\left(1+ \left(\frac{x}{t^{1/\alpha}}\right)^{\frac{2-\alpha}{2\alpha-2}}
 \right)\exp\left(-c_\alpha\left(\frac{x}{t^{1/\alpha}}\right)^{\alpha'}\right) 
\end{equation}
 and
 \begin{equation} \label{eq:12}
   c_1 (1+x/t^{1/\alpha} )^{-\alpha-1} \leq p_t(x) \leq c_2 (1+x)^{-\alpha-1} .
 \end{equation}
In particular, note that for every fixed $\eta\in (0,1)$ and
$x_0>0$, we have, for any $t\in (0,1)$ and $x\geq x_0$, 
\begin{equation}
  \label{eq:11bis}
  p_t(-x) \leq 
 C(\eta,x_0)\exp\left(-(1-\eta)
   c_\alpha\left(\frac{x}{t^{1/\alpha}}\right)^{\alpha'}\right). 
\end{equation}
A similar bound holds for $q_x(t)=\frac{x}{t}p_t(-x)$, with possibly
different constants.

Next, by \cite[Formula (3.20)]{Monrad-Silverstein}, it holds that 
there exists a positive constant $\eps_0>0$
 such that for all $\eps\in (0,\eps_0)$ and  for $0<a<b<+\infty$, 
 \begin{equation}
   \label{eq:13}
    c_1 \frac{\eps^{\alpha+1}}{a^\alpha t^{\frac{1-\alpha}{\alpha}}} \left(1 - \left(\frac{a}{b}\right)^\alpha \right) 
 \leq \int_{a/\eps}^{b/\eps} j_{t}(y)\, \d y \leq c_2
 \frac{\eps^{\alpha+1}}{a^\alpha t^{\frac{1-\alpha}{\alpha}}} \left(1
   - \left(\frac{a}{b}\right)^\alpha \right) .
\end{equation}

These estimates will allow us to evaluate the densities involved in
\eqref{eq:marginal-laws} in the large deviation regime. First we give
an explicit formula for $p^{(0,\infty)}_t(x,y)$. 
 
 \begin{lemma}\label{lem:error-p_t}
 Let $t>0$ and $x,y>0$. Then
 \begin{equation}\label{eq:error-p_t}
 p^{(0,\infty)}_t(x,y) = p_t(y-x) - \int_0^t q_x(s) p_{t-s}(y) \,\d s .
 \end{equation}
 \end{lemma}
 
 \begin{proof}
 Let $f:\R\rightarrow\R_+$ be a measurable function. On the one hand,
 we have by definition that 
 $$ \bE_x^{(0,\infty)}\left[ f(L_t) \right] = \int_\R f(y) p_t^{(0,\infty)}(x,y) \,\d y . $$
 On the other hand, if we denote by $T_0$ the first hitting time of $0$ by $(L_t)_{t\geq0}$, and $L^{(s)}_{t-s} = L_{s+(t-s)}$ for $s<t$, we then have
 \begin{eqnarray*}
 \bE_x^{(0,\infty)}\left[ f(L_t) \right]
 & = & \bE_x\left[ f(L_t) \1{T_0>t} \right] \\
 & = & \bE_x\left[ f(L_t) \right] - \bE_x\left[ f(L_t) \1{T_0<t} \right] \\
 & = & \int_\R f(y) p_{t}(y-x) \d y - \bE_x\left[ \1{T_0<t} \E{f(L^{(T_0)}_{t-T_0}) \mid \Fil_{T_0}} \right] \\
 & = & \int_\R f(y) p_{t}(y-x) \d y - \int_0^t q_x(s) \E{f(L^{(s)}_{t-s})} \d s \\
 & = & \int_\R f(y) p_{t}(y-x) \d y - \int_0^t q_x(s) \int_\R f(y)p_{t-s}(y)\d y \d s \\
 & = & \int_\R f(y) \left\{ p_t(y-x) - \int_0^t q_x(s) p_{t-s}(y) \,\d s \right\} \d y ,
 \end{eqnarray*}
 where we used the Markov property in the third equality. Thus
 Equation \eqref{eq:error-p_t} follows.
 \end{proof}

Using the scaling properties of $p_t(x)$ and $q_x(t)$, we may deduce from
Lemma \ref{lem:error-p_t} a bound on the error when we approximate
$p^{(0,\infty)}_t(x,y)$ by $p_t(y-x)$, as follows. 
 
 \begin{lemma}\label{lem:bound-error}
 For any fixed $t>0$ and $\eta>0$, there exists $C=C(\eta,t)>0$ such that for
 every $x,y>0$, 
 \begin{equation}
 \int_0^t q_{x}(s) p_{t-s}\left(y\right) \d s 
 \leq C\exp\left(-c_\alpha(1-\eta) \left(\frac{x^\alpha}{t}\right)^{\frac{1}{\alpha-1}} \right)  .
\end{equation}
 \end{lemma}

 \begin{proof}
   Using the scaling relations for $p_t(x)$, we have
   \begin{align*}
     \int_0^tq_x(s)p_{t-s}(y)\d s&=\int_0^{t/2}\frac{x}{s}p_s(-x)p_{t-s}(y)\d
                               s+\int_{t/2}^{t}\frac{x}{s}p_s(-x)p_{t-s}(y)\d
     s\\
     &\leq \|p_{t/2}\|_\infty \int_0^{t/2}\frac{x}{s}p_s(-x)\d
                               s+ \|p_1\|_\infty\int_{t/2}^t \frac{\d
       s}{(t-s)^{1/\alpha}}\frac{x}{s}p_s(-x) ,
   \end{align*}
   and then, using  \eqref{eq:11}, we get
   \begin{multline*}
        \int_0^tq_x(s)p_{t-s}(y)\d s \leq \\
        c_2\|p_{t/2}\|_\infty\int_{0}^{t/2}\frac{x\d s}{(t/2)^{\alpha}s^{1+1/\alpha}}(1+(x/s^{1/\alpha})^{\frac{2-\alpha}{2\alpha-2}})\exp\left(-c_\alpha(x/s^{1/\alpha})^{\alpha'}\right) \\
         +c_2\|p_1\|_\infty \frac{x}{(t/2)^{1+1/\alpha}}\left(1+\left(\frac{x}{(t/2)^{1/\alpha}}\right)^{\frac{2-\alpha}{2\alpha-2}}\right)\exp\left(-c_\alpha\left(\frac{x}{t^{1/\alpha}}\right)^{\alpha'}\right)\int_{t/2}^t \frac{\d     s}{(t-s)^{1/\alpha}} .
   \end{multline*}
   The second term is of the desired form, while, by performing a
   change of variables $u=s^{-\alpha'/\alpha}$, it is straightforward
   to see that the first term is negligible compared to the second. 
 \end{proof}
 
 \begin{proof}[Proof of Proposition \ref{prop:LDP-finite-dim}]
Since $J_\sigma$ is a good rate function on $\R^n$,  \cite[Lemma
5]{Serlet} shows that it suffices to prove that, for every open subset
$G\subset \R^n$,
 \begin{equation}
 \lim_{\eps\downarrow0} \eps^{\alpha'} \log
 \Prob{\eps(\Lex_{t_1},\dots,\Lex_{t_n})\in G} = - \inf_G J_\sigma \, .
 \end{equation}
Since $J_\sigma$ is infinite on $\R^n\setminus \R_+^n$, it suffices to
consider open sets $G$ of $\R_+^n$ (with the induced topology), and
this is what we do from now on.
 For convenience let us write
 $$ \Psi_\eps(x_1,\dots,x_n) =
 \prod_{i=1}^{n-1} 
   p_{t_{i+1}-t_i}^{(0,\infty)}\Big(\frac{x_i}{\eps},\frac{x_{i+1}}{\eps}\Big)
 \times  q_{x_n/\eps}(1-t_n) .$$
 Using \eqref{eq:marginal-laws}, 
 we can write
 $$ \Prob{\eps\big(\Lex_{t_1},\dots,\Lex_{t_n}\big) \in G}
 = C \eps^n \int_G \d x_1 \dots \d x_n \, j_{t_1}\Big(\frac{x_1}{\eps}\Big) \Psi_\eps(x_1,\dots,x_n)  , $$
 where $C=C(\alpha)>0$ is a positive constant depending only
 on $\alpha$.

 We start with the lower bound. For a given $\delta>0$, there exists
 $(y_1,\ldots,y_n)\in G$ such that $ J_\sigma(y_1,\dots,y_n) \leq
 \inf_G J_\sigma + \delta  $, and we may assume without loss of
 generality that $y_1,\ldots,y_n$ are pairwise distinct and all lie in $(0,\infty)$. Then, there
 exists a 
 hypercube $\mathcal Q_\delta = \prod_{i=1}^n(a_i,b_i) \subseteq G$
 containing $(y_1,\ldots,y_n)$ such that
 the intervals $[a_i,b_i]\subset (0,\infty)$ are pairwise disjoint, and 
 such that for all $(x_1,\dots,x_n) \in \mathcal Q_\delta$, we have 
 $$  J_\sigma(x_1,\dots,x_n) \leq \inf_G J_\sigma + \delta  . $$

Let us now consider the terms
$p^{(0,\infty)}_{t_{i+1}-t_i}(x_i/\eps,x_{i+1}/\eps)$ involved
in the definition of $\Psi_\eps(x_1,\ldots,x_n)$, where
$x_1,\ldots,x_n\in Q_\delta$. 
Fix $\eta\in(0,1)$. 
From \eqref{eq:12}  and Lemmas \ref{lem:error-p_t} and \ref{lem:bound-error}, 
 for every $i$ such that $y_i>y_{i+1}$, we may bound
 \begin{multline*}
   p^{(0,\infty)}_{t_{i+1}-t_i}\left(\frac{x_i}{\eps},\frac{x_{i+1}}{\eps}\right)
\\ 
\geq
c_1\exp\left(-\frac{c_\alpha}{\eps^{\alpha'}}\left(\frac{(b_i-a_{i+1})^\alpha}{t_{i+1}-t_i}\right)^{\frac{1}{\alpha-1}}\right)
-C(\eta)
  \exp\left(-\frac{c_\alpha}{\eps^{\alpha'}}(1-\eta)\left(\frac{a_i^\alpha}{t_{i+1}-t_i}\right)^{\frac{1}{\alpha-1}}\right) ,
 \end{multline*}
 and for every $i$ such that $y_i<y_{i+1}$,
  \begin{align*}
   p^{(0,\infty)}_{t_{i+1}-t_i}\left(\frac{x_i}{\eps},\frac{x_{i+1}}{\eps}\right)
\geq
c_1 \Big(1+\frac{b_{i+1}-a_i}{\eps(t_{i+1}-t_i)^{1/\alpha}}\Big)^{-\alpha-1}
-C(\eta)
  \exp\left(-\frac{c_\alpha}{\eps^{\alpha'}}(1-\eta)\left(\frac{a_i^\alpha}{t_{i+1}-t_i}\right)^{\frac{1}{\alpha-1}}\right) . 
 \end{align*}
 Also, we may bound
 $$q_{x_n/\eps}(1-t_n)\geq
 c_1
 \frac{a_n}{1-t_n}\exp\left(-\frac{c_\alpha}{\eps^{\alpha'}}\left(\frac{b_n^\alpha}{1-t_n}\right)^{\frac{1}{\alpha-1}}\right)
.$$
Therefore, by choosing $\eta$ small enough so that
$(1-\eta) a_i^{\alpha'}\geq (b_{i}-a_{i+1})^{\alpha'}$ for every $i$ such that $y_i>y_{i+1}$, we see that for every $\eps$ small enough,
$\Prob{(\Lex_{t_1},\ldots,\Lex_{t_n})\in G}$ is bounded from below by a
quantity of the form
$$ c\eps^n \tilde{\Psi}_\eps\int_{a_1}^{b_1}j_{t_1}\left(\frac{x_1}{\eps
  }\right)\d x_1 $$
where, for some constant $c$ depending only on $t_1,\ldots,t_n$,
$a_1,\ldots,a_n$ and $b_1,\ldots,b_n$, and $\eta$, 
$$\tilde{\Psi}_\eps=c\prod_{i=1}^n\left(\exp\left(-\frac{c_\alpha}{\eps^{\alpha'}}\left(\frac{(b_i-a_{i+1})^\alpha}{t_{i+1}-t_i}\right)^{\frac{1}{\alpha-1}}\right)\1{y_i>y_{i+1}}+\eps^{\alpha+1}\1{y_i<y_{i+1}}\right) , $$
with the convention that $y_{n+1}=a_{n+1}=0$.
 Using the asymptotics \eqref{eq:13}, we finally obtain
 $$ \liminf_{\eps\downarrow0} \eps^{\alpha'} \log
 \Prob{\eps(\Lex_{t_1},\dots,\Lex_{t_n})\in G} \geq -  c_\alpha\sum_{i=1}^n
\left(\frac{(b_i-a_{i+1})_+^\alpha}{t_{i+1}-t_i}\right)^{\frac{1}{\alpha-1}}. $$
 By letting $a_i$ and $b_i$ tend to $y_i$, we obtain 
 $$ \liminf_{\eps\downarrow0} \eps^{\alpha'} \log
 \Prob{\eps(\Lex_{t_1},\dots,\Lex_{t_n})\in G} \geq
 -J_\sigma(y_1,\ldots,y_n)\geq - \inf_G J
 -\delta , $$
 and since $\delta$ was arbitrary, we may conclude that 
 $$ \liminf_{\eps\downarrow0} \eps^{\alpha'} \log
 \Prob{\eps(\Lex_{t_1},\dots,\Lex_{t_n})\in G} \geq - \inf_G J . $$

 The corresponding upper bound is obtained by similar arguments.
 It is
 clear that the upper bound holds if $\inf_G J = 0$, so that we may
 assume $\inf_G J >0$. By Lemma \ref{lem:error-p_t}, we have 
 $p^{(0,\infty)}_t(x,y)\leq p_t(y-x)$. Therefore, 
 $$ \Psi_\eps(x_1,\dots,x_n) \leq C x_n \prod_{i=1}^{n-1} p_{t_{i+1}-t_i}\Big(\frac{x_{i+1}-x_i}{\eps}\Big) p_{1-t_n}\left(-\frac{x_n}{\eps}\right)  , $$
 where $C$ is a positive and finite constant that depends only on $t_1,\ldots,t_n$.

 Let $\eta\in (0,1/2)$ be a fixed constant.
 Now observe that for $t>0$, and $x\in \R$, we have
$$p_t(x)\leq p_t(x)\1{x<0}+\|p_t\|_\infty\1{x\geq 0}\leq
C\exp\left(-c_\alpha(1-\eta)\frac{(x_-)^{\alpha'}}{t^{\frac{1}{\alpha-1}}}\right)
,$$
where the constant $C$ depends only on $\eta$ and $t$, but not on
$x$. A similar bound holds for $xp_t(x)$, possibly with a different
constant $C$. 
 Thus we may write for all $(x_1,\dots,x_n) \in G$, with our usual
 convention that $x_{n+1}=0$ and $t_{n+1}=1$, and for a constant $C$
 that depends on $\eta,t_1,\ldots, t_n$ but not on $x_1,\ldots,x_n$, 
 \begin{align*}
 \Psi_\eps(x_1,\dots,x_n)
 \leq & C\exp\left(-\frac{(1-\eta)}{\eps^{\alpha'}}J_\sigma(x_1,\ldots,x_n)\right) \\
 \leq & C \exp\left( -\frac{1-2\eta}{\eps^{\alpha'}} \inf_G J \right)
        \exp\left( -\frac{\eta}{\eps^{\alpha'}}
        J_\sigma(x_1,\dots,x_n) \right) .
 \end{align*}
 Since $j_{t_1}\in\Lp^\infty(\R_+)$, we obtain after changing
 $x_i/\eps$ into $x_i$ in the integral, 
 \begin{align*}
 \Prob{\eps(\Lex_{t_1},\dots,\Lex_{t_n}) \in G} 
   \leq & C \exp\left( -\frac{1-2\eta}{\eps^{\alpha'}} \inf_G J \right)
          \int_{\R_+^n} \d x_1 \dots \d x_n \,  \exp\big( -\eta
          J_\sigma(x_1,\ldots,x_n) \big) 
 \end{align*}
and the last integral is finite. This implies that 
 $$ \limsup_{\eps\downarrow0} \eps^{\alpha'} \log \Prob{\eps(\Lex_{t_1},\dots,\Lex_{t_n}) \in G} 
 \leq - (1-2\eta) \inf_G J  . $$
 Since this is true for all $\eta>0$, we get the upper bound
 \[ \limsup_{\eps\downarrow0} \eps^{\alpha'} \log \Prob{\eps(\Lex_{t_1},\dots,\Lex_{t_n}) \in G} 
 \leq - \inf_G J  . \qedhere \]
 \end{proof}

\section{Exponential tightness for the normalized excursion}\label{section:expo-tight}

 In this section we prove the following proposition. 
 
 \begin{proposition}[Exponential tightness of $\Lex$]\label{prop:tightness}
 Under $\bP$, the laws of $(\eps \Lex_t)_{t\in[0,1]}$ as
 $\eps\downarrow0$ are exponentially tight in $(\D[0,1],\mathrm{dist})$ with speed $\eps^{-\alpha'}$.
\end{proposition}

In order to prove this result, we want to apply Theorem
\ref{sec:basic-results-skor-2}. Note that we already have exponential
tightness for $(\eps \Lex(q),\eps>0)$ for every $q\in [0,1]$, as a
consequence of Proposition \ref{prop:LDP-finite-dim} for $n=1$. It
turns out, however, that the criterion given in Theorem
\ref{sec:basic-results-skor-2} cannot immediately be used to obtain
exponential tightness over the
whole interval $[0,1]$. We must instead treat the intervals
$[0,1-\delta]$ and $[1-\delta,1]$ separately. 
 
 \begin{lemma}\label{sec:expon-tightn-norm}
 For every $\delta\in(0,1)$, there exists a constant $C=C(\alpha,\delta)\in(0,\infty)$ such that for every $s,t,u\in[0,1-\delta]$ with $s\le t\le u$ and $\lambda\ge0$,
 $$ \E{\exp\big(\lambda M(\Lex_s,\Lex_t,\Lex_u)\big)}
  \le C \exp\big((u-s)\lambda^\alpha\big). $$
 \end{lemma}
 
 \begin{proof}
 We split the expectation into five terms: 
 \begin{align*}
 \E{\exp\big(\lambda M(\Lex_s,\Lex_t,\Lex_u)\big)}
 & = \E{e^{\lambda(\Lex_s-\Lex_t)} \1{\Lex_t\le\Lex_s\le\Lex_u}} + \E{e^{\lambda(\Lex_t-\Lex_u)} \1{\Lex_s\le\Lex_u\le\Lex_t}} \\
 & \quad + \E{e^{\lambda(\Lex_t-\Lex_s)} \1{\Lex_u\le\Lex_s\le\Lex_t}} + \E{e^{\lambda(\Lex_u-\Lex_t)} \1{\Lex_t\le\Lex_u\le\Lex_s}} \\
 & \quad + \Prob{\{\Lex_s\le\Lex_t\le\Lex_u\}\cup\{\Lex_u\le\Lex_t\le\Lex_s\}} \\
 & \le 2\E{e^{\lambda(\Lex_s-\Lex_t)} \1{\Lex_t\le\Lex_s}} + 2\E{e^{\lambda(\Lex_t-\Lex_u)} \1{\Lex_u\le\Lex_t}} +1 .
 \end{align*}
 We see that the two expectation terms on the last line are of the same form $\E{e^{\lambda(\Lex_a-\Lex_b)} \1{\Lex_b\le\Lex_a}}$ where $a\le b$ with $b-a\le u-s$. For such $a,b$, letting $c=\alpha\Gamma(1-1/\alpha)$, we have
 \begin{align*}
 \E{e^{\lambda(\Lex_a-\Lex_b)} \1{\Lex_b\le\Lex_a}}
 & = c\int_0^\infty \mathrm dx \int_0^\infty \mathrm dy \: j_a(x) p^{(0,\infty)}_{b-a}(x,y) q_y(1-b) e^{\lambda(x-y)} \1{y\le x} \\
 & \le c\int_0^\infty \mathrm dz \: p_{b-a}(-z) e^{\lambda z} \int_z^\infty j_a(x) q_{x-z}(1-b) \d x ,
 \end{align*}
 where we have used the fact that $p^{(0,\infty)}_{b-a}(x,y)\le p_{b-a}(y-x)$ and a change of variables.
 We claim that the last integral in $x$ is uniformly bounded over $z\ge0$, $0\le a<b\le 1-\delta$.
 If we can prove this claim, then this will imply the existence of a finite constant such that
 $$ \E{e^{\lambda(\Lex_a-\Lex_b)} \1{\Lex_b\le\Lex_a}}
 \le C\E{e^{-\lambda L_{b-a}}} = C\exp\big((b-a)\lambda^\alpha\big)
 \leq C\exp\big((u-s)\lambda^\alpha\big) , $$
 for every $a,b\in[0,1-\delta]$ with $a\le b$ and $b-a\leq u-s$, which gives the result.
 \newline
 To prove the claim, note that
 \begin{align*}
 \int_z^\infty j_a(x) q_{x-z}(1-b) \d x
 & = \int_z^\infty j_a(x) \frac{x-z}{1-b}p_{1-b}(z-x) \d x \\
 & \le \frac{\Vert p_\delta\Vert_\infty}{\delta} \int_0^\infty xj_a(x) \d x \\
 & = \frac{ \Vert p_\delta\Vert_\infty }{\delta}\int_0^\infty x j_1(x) \d x,
 \end{align*}
 where in the last display we have used the scaling relation \eqref{eq:3} that implies that the integral does not depend on $a$.
 Letting $\bar J_1(x) = \int_x^\infty j_a(y)\d y$, we may integrate by
 parts and get an upper bound which is within a multiplicative
 constant of 
 $$ \big[x \bar J_1(x)\big]_0^\infty + \int_0^\infty \bar J_1(x) \d x . $$
 Now by \cite[(3.20)]{Monrad-Silverstein}, we have that $\bar J_1(x)
 \sim \bar c x^{-\alpha}$ as $x\to\infty$ for some finite constant
 $\bar c$, and the desired uniform upper bound follows.
 \end{proof}
 
 Our next lemma shows that $\Lex$ is exponentially well-behaved near time 1.
 
 \begin{lemma}\label{sec:expon-tightn-norm-1}
 For every $\lambda,\gamma>0$, there exists $\delta\in(0,1)$ such that
 $$ -\limsup_{\eps\downarrow0} \eps^{\alpha'} \log \Prob{\sup_{1-\delta\le t\le1} \eps\Lex_t\ge\gamma} \ge \lambda . $$
 \end{lemma}

 \begin{proof}
 For all $\delta>0$, we have
 $$ \Prob{\sup_{1-\delta\le t\le 1} \eps\Lex_t\ge \gamma}
 = \Prob{\eps\Lex_{1-\delta}\ge\gamma} + \Prob{\sup_{1-\delta<t\le 1} \eps\Lex_t\ge\gamma, \eps\Lex_{1-\delta}<\gamma}. $$
 From Proposition \ref{prop:LDP-finite-dim}, 
 \begin{equation}\label{eq:1st-prob}
 \limsup_{\eps\downarrow 0}\eps^{\alpha'}\log \Prob{\eps\Lex_{1-\delta}\geq \gamma}
 \leq
 -c_\alpha\left(\frac{\gamma^\alpha}{\delta}\right)^{\frac{1}{\alpha-1}} .
 \end{equation}
 Let us prove now a similar bound for the second probability.
 By \eqref{eq:15}, we may recast it as 
 \begin{equation}\label{eq:law-n}
  \Prob{\sup_{1-\delta<t\le 1} \eps\Lex_t\ge\gamma,
    \eps\Lex_{1-\delta}<\gamma}= \alpha\Gamma\left(1-\frac{1}{\alpha}\right)
  \int_0^{\gamma/\eps} \mathrm{d} x \, j_{1-\delta}(x) q_x(\delta)
  \bP^\delta_x\left(\sup_{0\leq t\leq \delta}L\geq \gamma/\eps\right) .
  \end{equation}
Now note that $\bP^\delta_x(\sup_{0\leq t\leq \delta}L\geq \gamma/\eps)$
is the limit of $\bP^\delta_x(\sup_{0\leq t\leq \delta'}L\geq
\gamma/\eps)$ as $\delta'\uparrow \delta$. By the absolute continuity
relation~\eqref{eq:16} and an elementary martingale argument, the
latter can be rewritten as
$$\bE_x\left[\1{T_0>S}\1{S<\delta'}\frac{q_{L_S}(\delta-S)}{q_x(\delta)}\right]
,$$
where $S$ denotes the stopping time $\inf\{t\geq 0:
L_t>\gamma/\eps\}$. Finally, for every $\eta\in (0,1)$, we may use
\eqref{eq:11bis} to obtain 
$$ q_{\gamma/\eps}(\delta-S) \le C(\eta,\gamma) \exp\bigg( -(1-\eta) c_\alpha \left(\frac{\gamma^\alpha}{\eps^\alpha \delta}\right)^{\frac{1}{\alpha-1}}\bigg). $$
Since this bound does not depend on $\delta'$, plugging it into the
previous expectation gives
 \begin{multline*}
  \Prob{\sup_{1-\delta<t\le 1} \eps\Lex_t\ge\gamma,
    \eps\Lex_{1-\delta}<\gamma} 
    \\ \leq 
     \alpha\Gamma\left(1-\frac{1}{\alpha}\right)
  \int_0^{\gamma/\eps} \mathrm{d} x \, j_{1-\delta}(x) C(\eta,\gamma)
 \exp\bigg( -(1-\eta) c_\alpha \left(\frac{\gamma^\alpha}{\eps^\alpha \delta}\right)^{\frac{1}{\alpha-1}}\bigg). 
  \end{multline*}
Since $j_{1-\delta}$ is integrable, the desired bound follows. 
 \end{proof}

 \begin{proof}[Proof of Proposition \ref{prop:tightness}]
 Fix $\lambda>0$. By Lemma \ref{sec:expon-tightn-norm-1}, for every $n\ge1$, there exists a $\delta_n\in(0,1)$ such that for every $\eps>0$ small enough
 $$ \Prob{\sup_{1-2\delta_n<t\le 1}\eps\Lex\ge\frac{1}{2^n}} \leq
 \frac{\exp\big(-\lambda\eps^{-\alpha'}\big)}{2^n} . $$
For this choice of $\delta_n$, by Lemma \ref{sec:expon-tightn-norm} and Theorem \ref{sec:basic-results-skor-2}, for every $n\ge0$ there exists a compact set $K^{(n)}_\lambda$ of $\D[0,1]$ such that for every $\eps>0$ small enough,
 $$ \Prob{(\eps\Lex^{(n)}_t,0\le t\le 1-\delta_n)\not\in
   K^{(n)}_\lambda} \leq \frac{\exp\big(-\lambda
   \eps^{-\alpha'}\big)}{2^n} , $$
 where $\Lex^{(n)}$ is the process $(\Lex_{t\wedge (1-\delta_n)},0\leq
 t\leq 1)$. 
 We conclude by noting that the set $K_\lambda$ of functions $f\in \D[0,1]$
 such that for every $n\ge0$, $\big(f(t\wedge (1-\delta_n)),0\le t\le
 1\big)\in K^{(n)}_\lambda$ and $\sup_{1-2\delta_n\le t\le 1}|f(t)|\le
 2^{-n}$ is relatively compact, and, by the above, satisfies $\Prob{\eps\Lex\notin
 K_\lambda}\leq 2\exp(-\lambda\eps^{-\alpha'})$. 
 \end{proof}

\section{Proof of Theorem \ref{thm:LDP-excursion}}\label{section:thm1}

 In this section, we prove Theorem \ref{thm:LDP-excursion}.
We will do this by combining the exponential tightness with a weaker
form of the {\fontfamily{lmss}\selectfont LDP}, as we now explain. The {\em weak
   topology} $\mathcal{W}$ on $\D[0,1]$ is the topology generated by the basis of
 neighborhoods of the form
 $$N(f,t_1,\ldots,t_k,\eps_1,\ldots,\eps_k)
 =\big\{g\in \D[0,1]:|g(t_i)-f(t_i)|<\eps_i,1\leq i\leq k\big\}\, ,$$
 where $f\in \D[0,1]$,  $\eps_1,\ldots,\eps_k>0$, and $t_1,\ldots,t_k$ are
 elements of $[0,1]$ {\em that are continuity points of $f$}. Here,
 by convention, $0$ is a continuity point of $f$ if and only if
 $f(0)=0$, which is consistent with our convention that $f(0-)=0$.
Clearly this defines a Hausdorff topology, since two different
elements of $\D[0,1]$ necessarily differ at some common continuity point. 
It is easy to see that a sequence $(f_n,n\geq 0)$ that converges to a limit $f$ in
 $(\D[0,1],\mathrm{dist})$ also converges pointwise at every
 continuity point of $f$, and, therefore, will eventually belong to
 any given basic neighborhood
 $N(f,t_1,\ldots,t_k,\eps_1,\ldots,\eps_k)$ for the topology $(\D[0,1],\mathcal{W})$. 
Since open sets in the metric space $(\D[0,1],\mathrm{dist})$ can be
characterized sequentially, this implies that 
 the weak topology is coarser than the topology of
 $(\D[0,1],\mathrm{dist})$. 
Therefore, by \cite[Corollary 4.2.6]{Dembo-Zeitouni}, and by the
exponential tightness established 
 in Proposition \ref{prop:tightness}, Theorem \ref{thm:LDP-excursion} will follow from
 the following statement.
 \begin{proposition}
   \label{sec:proof-crefthm:ldp-ex}
   The laws of $(\eps\Lex_t)_{t\in[0,1]}$ satisfy an {\fontfamily{lmss}\selectfont LDP} in
$(\D[0,1],\mathcal{W})$ as $\eps\downarrow0$ with speed $\eps^{-\alpha'}$
and good rate function $I_\Lex$. 
 \end{proposition}
The remainder of this section is thus devoted to the proof of this
proposition, which follows the approach of 
\cite{Lynch} closely.

 \subsection{Facts about the rate
   function}\label{sec:facts-about-rate}

 Denote by $\mathfrak{S}$ the set of finite subdivisions of $[0,1]$. 
 For
 $\sigma = (t_1,\ldots,t_n)\in\mathfrak{S}$, where $0<t_1<\cdots<t_n<1$, recall
 that for $x_1,\dots,x_n \in \R_+$, we let 
 $$ J_\sigma(x_1,\dots,x_n) = c_\alpha \sum_{i=1}^{n} (t_{i+1}-t_i) \left(\frac{(x_i-x_{i+1})_+}{t_{i+1}-t_i}\right)^{\alpha'} $$
 where, by convention, $x_{n+1}=0$ and $t_{n+1}=1$. We let
 $J_\sigma(x_1,\ldots,x_n)=\infty$ if one of the $x_i$'s is negative. 
 To ease notation, for $f:[0,1]\to \R$, we let 
$I_\Lex^\sigma(f)=J_\sigma(f(t_1),\ldots,f(t_n))$. 
 By the Dawson-G\"artner theorem (see \cite[Theorem 4.6.1]{Dembo-Zeitouni}), 
 it follows from Proposition \ref{prop:LDP-finite-dim} that
 the laws of $(\eps \Lex_t)_{t\in[0,1]}$ satisfy an
 {\fontfamily{lmss}\selectfont LDP} in $\R^{[0,1]}$ (with the product topology) as $\eps \downarrow 0$, with speed
 $\eps^{-\alpha'}$ and good rate function 
 \begin{equation}\label{formula I_chi}
 \widetilde{I}_\Lex(f) = \sup_{\sigma \in \mathfrak{S}} J_\sigma ( f(t_1),\dots, f(t_n) ).
 \end{equation}
We cannot immediately make use of this, since the domain of this rate function is not a space of c\`adl\`ag functions.  However, let us
prove some properties of the rate function $\widetilde{I}_\Lex$ and, in
particular, that its restriction to $\D[0,1]$ coincides with the rate function $I_\Lex$ given in Theorem 
 \ref{thm:LDP-excursion}. To this end, we prove the following
 proposition.

 \begin{proposition}\label{prop:criteria-Hex}
 A function $f\in \D[0,1]$ with $f\geq 0$ and $f(1)=0$ is in $\Hex$ if and only if
 $$ M_{\Lex}(f) = \sup_{\sigma \in \mathfrak{S}} I_\Lex^\sigma(f)< \infty . $$
 In this case, we have
 $$ M_{\Lex}(f) = I_\Lex(f)\, $$
 and, consequently, the functions $I_\Lex$ and $\widetilde{I}_\Lex$
 coincide on $\D[0,1]$. 
 \end{proposition}
 
 \begin{proof}
This statement should be compared with \cite[Theorem 3.2]{Lynch},
where the proof uses a martingale argument. We provide another elementary proof
here, based on the Lebesgue differentiation theorem instead. 
   For convenience, let $\Lambda(x) = c_\alpha(x_-)^{\alpha'}$ for all $x\in\R$.
 
 Let $f\in\Hex$, and write $f=\fup-\fdo$ for its  Jordan
 decomposition with absolutely continuous  $\fdo$, such that $\fdo'\in \Lp^{\alpha'}[0,1]$. 
 Let $\sigma = (t_1,\ldots,t_n)$ be a subdivision of $[0,1]$. Here and
 below, we adopt the notational convention that $t_0=0$ and $t_{n+1}=1$.
 Then
 \begin{eqnarray*}
 c_\alpha\sum_{i=0}^{n} \left( \frac{\big( f(t_i)-f(t_{i+1}) \big)_+^\alpha}{t_{i+1}-t_i} \right)^{\frac{1}{\alpha-1}}
 & = & \sum_{i=0}^{n} (t_{i+1}-t_i) \Lambda\left( \frac{f(t_{i+1})-f(t_i)}{t_{i+1}-t_i} \right) \\
 & \leq & \sum_{i=0}^{n} (t_{i+1}-t_i) \Lambda\left( \frac{\fdo(t_i)-\fdo(t_{i+1})}{t_{i+1}-t_i}  \right) \\
 & = & c_\alpha\sum_{i=0}^{n} (t_{i+1}-t_i) \left( \frac{1}{t_{i+1}-t_i} \int_{t_i}^{t_{i+1}} \fdo'(s) \d s \right)^{\alpha'} \\
 & \leq & c_\alpha\sum_{i=0}^{n} \int_{t_i}^{t_{i+1}} \fdo'(s)^{\alpha'} \d s \\
 & = & c_\alpha\int_0^1 \fdo'(s)^{\alpha'} \d s ,
 \end{eqnarray*}
 where we used the fact that $\fup$ 
 and $\Lambda$ are 
 non-increasing in the first inequality, and applied Jensen's inequality 
 in the second inequality. 
 Since this is true for any subdivision $\sigma\in \mathfrak{S}$, we get the first bound $M_{\Lex}(f) \leq \int_0^1 \fdo'(s)^{\alpha'} \, \textrm{d}s < \infty$.
 
 Conversely, assume that $f\in \Hex$ is not of bounded variation and  fix $A>0$. 
 Then there exists a subdivision $\sigma=(t_1,\ldots,t_n)$ such that
 $$ A
 < \sum_{i=0}^n \left| f(t_{i+1}) - f(t_i) \right|
 = \sum_{i=0}^n \big( f(t_i) - f(t_{i+1}) \big)_+ + \sum_{i=1}^n \big( f(t_i) - f(t_{i+1}) \big)_- . $$
 Furthermore
 $$ f(0)
 = \sum_{i=0}^n \big( f(t_i) - f(t_{i+1}) \big)
 = \sum_{i=0}^n \big( f(t_i) - f(t_{i+1}) \big)_+ - \sum_{i=0}^n \big( f(t_i) - f(t_{i+1}) \big)_-. $$
 This implies that
 $$ \sum_{i=0}^n \big( f(t_i) - f(t_{i+1}) \big)_+ \geq \frac{A + f(0)}{2} . $$
Therefore,
\begin{multline*}
 \frac{A+f(0)}{2}\leq
 \sum_{i=0}^n\big(f(t_i)-f(t_{i+1})\big)_+=\sum_{i=0}^n(t_{i+1}-t_i)^{1/\alpha}\frac{\big(f(t_i)-f(t_{i+1})\big)_+}{(t_{i+1}-t_i)^{1/\alpha}}\\
   \leq \left(\sum_{i=0}^n(t_{i+1}-t_i)\right)^{1/\alpha}
   \left(\sum_{i=0}^n
     \frac{\big(f(t_i)-f(t_{i+1})\big)_+^{\alpha'}}{(t_{i+1}-t_i)^{\frac{1}{\alpha-1}}}\right)^{1/\alpha'}\, ,
 \end{multline*}
 by Hölder's inequality, and 
 this entails that $M_{\Lex}(f)=\infty$. 
 By the contrapositive, this implies that if $M_{\Lex}(f)<\infty$, then $f$ has bounded variation. 
 Therefore, assuming that $M_{\Lex}(f)<\infty$, we may write
 $f=\fup-\fdo$ for the Jordan decomposition of $f$,  with $\fup,\fdo$
 nondecreasing and such that  $\fup(0)=0$ and $\mathrm d\fup \perp \mathrm d\fdo$. 
We proceed by contradiction. Suppose that $\fdo$ is not absolutely continuous. 
 Then there exists $\eps>0$ such that for all $k \geq 1$, there exists
 an open set of the form $U_k=\bigsqcup_{i=1}^{n(k)}(s_i^{(k)},t_i^{(k)})$
with
 $$ \sum_{i=1}^{n(k)} ( t^{(k)}_i - s^{(k)}_i ) < \frac{1}{k} \quad \textrm{ and } \quad \sum_{i=1}^{n(k)} \big(( \fdo( t^{(k)}_i ) - \fdo( s^{(k)}_i ) \big) > 2\eps . $$
 Moreover since $\mathrm d\fdo \perp \mathrm d\fup$, there exists a
 measurable set $B$ such that $\d \fdo(B^c)=\d \fup(B)=0$, and by
 regularity of the measures $\mathrm{Leb}, \d \fdo$ and $\d \fup$
 applied to the set $B\cap U_k$, we
 may find open sets $V_k^{(1)},V_k^{(2)}$ containing $B\cap
 U_k$ such that
 $$\mathrm{Leb}(V_k^{(1)})<\frac{1}{k}\,  \qquad \mbox{and}\qquad  \d
 \fup(V_k^{(2)})<\eps ,$$
 and by setting $V_k=V^{(1)}_k\cap V^{(2)}_k$, we see that these two
 inequalities remain true with $V_k$ in place of $V^{(1)}_k$ and
 $V^{(2)}_k$ respectively, while 
 $$\d\fdo(V_k)\geq \d \fdo(B\cap
 U_k)=\d\fdo(U_k)>2\eps .$$
By writing the open set $V_k$ as the limit of finite unions of open
intervals, we deduce that we may choose the family of intervals $\{
(s_i^{(k)},t_i^{(k)}), \, 1 \leq i \leq n(k) \}$ in such a way that
 $$ \sum_{i=1}^{n(k)} \big( \fup( t^{(k)}_i ) - \fup( s^{(k)}_i ) \big) < \eps . $$
 Then on the one hand we have
 \begin{eqnarray}
 \sum_{i=1}^{n(k)} \big( f( s^{(k)}_i ) - f( t^{(k)}_i ) \big)_+
 & = & \sum_{i=1}^{n(k)} \big( \fdo( t^{(k)}_i ) - \fdo( s^{(k)}_i ) - ( \fup( t^{(k)}_i ) - \fup( s^{(k)}_i ) ) \big)_+ \nonumber \\
 & \geq &  \sum_{i=1}^{n(k)} \big( \fdo( t^{(k)}_i ) - \fdo( s^{(k)}_i ) \big) - \sum_{i=1}^{n(k)} \big( \fup( t^{(k)}_i ) - \fup( s^{(k)}_i ) \big) \nonumber \\
 & \geq & \eps . \label{eq:h-abs-continuous1}
 \end{eqnarray}
 On the other hand, by H\"older's inequality, 
 \begin{eqnarray}
 \sum_{i=1}^{n(k)} \big( f( s^{(k)}_i ) - f( t^{(k)}_i ) \big)_+
 & \leq & \left( \sum_{i=1}^{n(k)} (t^{(k)}_i - s^{(k)}_i) \right)^{\frac{1}{\alpha}} \left( \sum_{i=1}^{n(k)} \frac{\big( f(s^{(k)}_i) - f(t^{(k)}_i) \big)_+^{\alpha'}}{(t^{(k)}_i - s^{(k)}_i)^{\frac{1}{\alpha-1}}} \right)^{1/\alpha'} \nonumber \\
 & \leq & \frac{1}{c_\alpha k^{1/\alpha}} M_{\Lex}(f). \label{eq:h-abs-continuous2}
 \end{eqnarray}
 But \eqref{eq:h-abs-continuous1} and \eqref{eq:h-abs-continuous2}
 combined contradict the assumption that $M_{\Lex}(f)<+\infty$. 
 Thus $\fdo$ is absolutely continuous. 

 Let us now prove that $\int_0^1 \fdo'(s)^{\alpha'} \d s \leq M_{\Lex}(f)$,
 which will prove that $f \in \Hex$, and that if $f \in \Hex$, then $M_{\Lex}(f) = \int_0^1 \fdo'(s)^{\alpha'} \d s$. 
 For $n\geq1$, define
 $$ f^{(n)}(t) = n \left( f\left( \frac{\left\lfloor (n+1)t \right\rfloor }{n} \right) - f\left( \frac{\left\lfloor nt \right\rfloor }{n} \right) \right) \textrm{ for } t\in[0,1), \quad f^{(n)}(1) = n \left( f( 1 ) - f\left( 1-\frac{1}{n} \right) \right). $$
By the Lebesgue decomposition theorem, we may write $\fup = \fupac + \fups$, where $\fupac$ is an absolutely
continuous function and $\fups$ is such that $\d\fups$ is singular with
respect to the Lebesgue measure. By the Lebesgue differentiation theorem, for Lebesgue-almost every $t\in[0,1]$ we have
 $$ f^{(n)}(t) \underset{n\rightarrow+\infty}{\longrightarrow} \fupac'(t) - \fdo'(t) .$$
 Considering the subdivision $\sigma_n = \left(0,\frac{1}{n},\frac{2}{n},\dots,1\right)$, we then have
 \begin{eqnarray*}
 M_{\Lex}(f) & \geq & \liminf_{n\rightarrow+\infty} J_{\sigma_n}(f) \\
 & = & \liminf_{n\rightarrow+\infty} \sum_{i=1}^n \frac{1}{n} \Lambda\left( f^{(n)}\left(\frac{i}{n}\right) \right) \\
 & = & \liminf_{n\rightarrow+\infty} \int_0^1 \Lambda\left( f^{(n)}(s) \right) \d s .
 \end{eqnarray*}
 By Fatou's lemma, we thus get
 $$ M_{\Lex}(f) \geq  \int_0^1 \Lambda\left( \fupac'(s) - \fdo'(s) \right) \d s . $$
 Recall that $\mathrm{d}\fup \perp \mathrm{d}\fdo$, so that we also have
 $\d\fupac\perp \d \fdo$.
But by the 
Lebesgue differentiation theorem, we have $\d\fupac(s)=\fupac'(s)\d s$
and $\d \fdo(s)=\fdo'(s)\d s$, so that the sets $\{\fupac'>0\}$ and
$\{\fdo'>0\}$ intersect in a set of zero Lebesgue measure. 
 Since $\fupac' \geq 0$ Lebesgue-a.e.,
 we have $\Lambda( \fupac' ) = 0$, 
which implies that
 $$ \int_0^1 \Lambda\left( \fupac'(s) - \fdo'(s) \right) \d s = \int_0^1 \Lambda\left( -\fdo'(s) \right) \, \textrm{d}s = c_\alpha\int_0^1 \fdo'(s)^{\alpha'} \d s, $$
 which concludes the proof.
 \end{proof}

In passing, we note that the reasoning at the end of this proof explains
why we may express the rate function $I_\Lex$ in the alternative form
\eqref{eq:17}. 

\begin{lemma}
  \label{sec:facts-about-rate-1}
  The function $I_\Lex$ is a good rate function on the spaces
  $(\D[0,1],\mathcal{W})$ and $(\D[0,1],\mathrm{dist})$. 
\end{lemma}

\begin{proof}
Since the weak topology is not first-countable, we must initially use nets to characterise the lower-semicontinuity of $I_\Lex$. We first need to show that if $(f_\lambda)$ is a net that converges to $f$ in the weak topology, with
  $f_\lambda\geq 0$ and $f_\lambda(1)=f(1)=0$ for every $\lambda$, then
  $\liminf_\lambda\ I_\Lex(f_\lambda)\geq I_\Lex(f)$. Let $\sigma=(t_1,\ldots,t_k)$
  be a subdivision of continuity points of $f$, so that $f_\lambda(t_i)$ converges to 
 $f(t_i)$ for $1\leq i\leq k$, and 
  $I^\sigma_\Lex(f_\lambda)=J_\sigma(f_\lambda(t_1),\ldots,f_\lambda(t_k))$ converges to
  $I^\sigma_\Lex(f)=J_\sigma(f(t_1),\ldots,f(t_k))$. Since $I_\Lex(f_\lambda)\geq
  I^\sigma_\Lex(f_\lambda)$ by Proposition \ref{prop:criteria-Hex},
  this implies that $\liminf_\lambda I_\Lex(f_\lambda)\geq
  I^\sigma_\Lex(f)$. Applying Proposition \ref{prop:criteria-Hex}
  once again allows us to conclude that $I_\Lex$ is a rate function on
  $(\D[0,1],\mathcal{W})$, and therefore also on  $(\D[0,1],\mathrm{dist})$. 

Let us now prove that $I_\Lex$ is good on $(\D[0,1],\mathrm{dist})$,
which will imply the result.  Fix $c\in (0,\infty)$, and then pick $f\in \D[0,1]$ with
$I_\Lex(f)\leq c$ so that, in particular, $f(1)=0$ and $f$ has bounded variation. Let $s\leq t$ be in
$[0,1]$. Then, by Hölder's inequality, 
 $$ f(s)-f(t) \leq \fdo(t)-\fdo(s) = \int_s^t \fdo'(u) \d u\leq (c/c_\alpha) (t-s)^{1/\alpha}. $$
 Since $f(1)=0$, this implies that $f$ is uniformly bounded and, 
 moreover, that for every $s\leq t\leq u$ we have
 \[ M\big(f(s),f(t),f(u)\big) \leq 2\big((f(s)-f(t))\1{f(t)\leq
     f(s)}+(f(t)-f(u))\1{f(u)\leq f(t)}\big)\leq 4(c/c_\alpha) (u-s)^{1/\alpha}.\]
The conclusion now follows from Theorem \ref{sec:basic-results-skor-1}. 
\end{proof}

Next, for $A\subset \D[0,1]$, and for
$\sigma\in \mathfrak{S}$, we let
$$I_\Lex^\sigma(A)=\inf_{f\in A}I_\Lex^\sigma(f)\, \qquad \mbox{ and
}\qquad I_\Lex(A)=\inf_{f\in A}I_\Lex(f)\, .$$

\begin{lemma}
  \label{sec:facts-about-rate-2}
  For every closed set $F$ of $(\D[0,1],\mathcal{W})$, we have
  $$I_\Lex(F)=\sup_{\sigma\in \mathfrak{S}}I^\sigma_\Lex(F)\, .$$
\end{lemma}

\begin{proof}
The proof follows that of \cite[Theorem 3.5]{Lynch} closely. 
  Since we know from Proposition \ref{prop:criteria-Hex} that
  $I^\sigma_\Lex(A)\leq I_\Lex(A)$ for every $\sigma\in \mathfrak{S}$
  and every set $A$, let us assume, for a contradiction, that
  $\sup_{\sigma\in \mathfrak{S}}I^\sigma_\Lex(F)<c< I_\Lex(F)$ for some constant $c\in
  (0,\infty)$. For every subdivision $\sigma=(t_1,\ldots,t_k)\in \mathfrak{S}$, we may
  find some element $f_\sigma\in \D[0,1]$ such that
  $I^\sigma_\Lex(f_\sigma)<c$. Let $\hat{f}_\sigma$ be the
  piecewise affine interpolation of the values of $f_\sigma$ at times $0<t_1<\ldots<t_k<1$,  with 
  $\hat{f}_\sigma(0)=\hat{f}_\sigma(1)=0$ and of course 
  $\hat{f}_\sigma(t_i)=f_\sigma(t_i),1\leq i\leq k$. 
Then, plainly,
$$I_\Lex(\hat{f}_\sigma)=I^\sigma_\Lex(\hat{f}_\sigma)=I^\sigma_\Lex(f_\sigma)<c\, $$
and, by
Lemma \ref{sec:facts-about-rate-1}, we obtain that
$\{\hat{f}_\sigma,\sigma\in \mathfrak{S}\}$ forms a relatively compact set
 in $(\D[0,1],\mathcal{W})$ (and even in
$(\D[0,1],\mathrm{dist})$). Let $f_0$ be a cluster point of this
set, and $\sigma'=(t_1',\ldots,t_l')\in \mathfrak{S}$ be a
subdivision consisting of continuity points of $f_0$. We fix $\eps>0$
and consider the weak neighborhood of $f_0$ defined by 
$$N_{\sigma',\eps}=\Big\{f\in \D[0,1]:\max_{1\leq i\leq
  l}|f(t'_i)-f_0(t'_i)|<\eps\Big\}\in \mathcal{W}\, .$$
For any  partition $\sigma''$ finer than $\sigma'$, there exists an
even finer $\sigma$ such that $\hat{f}_\sigma\in N_{\sigma',\eps}$,
since $f_0$ is a cluster point. But since $\hat{f}_\sigma$ agrees with
$f_\sigma$ on $\sigma$, it follows that $f_\sigma\in
N_{\sigma',\eps}$ and, therefore, that $f_0$ is also a cluster point
of $\{f_\sigma:\sigma\in \mathfrak{S}\}\subset F$. Since $F$ is closed,
we conclude that  $f_0\in F$, and that $I_\Lex(f_0)\leq c$ by lower
semicontinuity of $I_\Lex$. This contradicts the assumption that
$I_{\Lex}(F)>c$, and the result follows.
\end{proof}

We now have all the tools needed to prove Proposition
\ref{sec:proof-crefthm:ldp-ex}. The proof is split into two lemmas which
follow \cite[Theorems 4.1 and 4.2]{Lynch} closely.

\begin{lemma}
  \label{sec:facts-about-rate-3}
  If $F$ is a closed subset of $(\D[0,1],\mathcal{W})$, then
  $$\limsup_{\eps\downarrow 0}\eps^{\alpha'}\log \Prob{\eps\Lex\in
    F}\leq -I_\Lex(F)\, .$$
\end{lemma}

\begin{proof}
  For any $\sigma=(t_1,\ldots,t_k)\in \mathfrak{S}$, it holds that
  $$\Prob{\eps\Lex\in F}\leq \Prob{I^\sigma_\Lex(\eps\Lex)\geq
    I^\sigma_\Lex(F)}\, .$$
  From the explicit form of $I^\sigma_\Lex(\eps\Lex)$, and the fact
  that $\eps( \Lex_{t_1},\ldots,\Lex_{t_k})$ satisfy an {\fontfamily{lmss}\selectfont LDP} with
  continuous rate function $J_\sigma$ by
  Proposition \ref{prop:LDP-finite-dim}, we obtain by the contraction
  principle that
  $$\limsup_{\eps\downarrow 0}\eps^{\alpha'}\log\Prob{\eps\Lex\in
  F}\leq -I^\sigma_\Lex(F)\, .$$
We conclude using Lemma \ref{sec:facts-about-rate-2}.    
\end{proof}

\begin{lemma}
  \label{sec:facts-about-rate-4}
  If $G$ is an open subset of $(\D[0,1],\mathcal{W})$, then
  $$\liminf_{\eps\downarrow 0}\eps^{\alpha'}\log \Prob{\eps\Lex\in
    G}\geq -I_\Lex(G)\, .$$
\end{lemma}

\begin{proof}
  Without loss of generality, we may assume that $I_\Lex(G)<\infty$. We
  then fix $\eps > 0$ and select $f\in G$ such that
  $I_\Lex(f)<I_\Lex(G)+\eps$. Then, we may find $\delta>0$ and a
  subdivision $\sigma=(t_1,\ldots,t_k)$ consisting of continuity points of $f$
  such that $\{g\in \D[0,1]:\max_{1\leq i\leq
    k}|g(t_i)-f(t_i)|<\delta\}$ is contained in $G$. We deduce that
  $$\Prob{\eps\Lex\in G}\geq \Prob{\eps
    (\Lex_{t_1},\ldots,\Lex_{t_k})\in G'}$$
  where $G'$ is the open set $\{(x_1,\ldots,x_k)\in \R^k:\max_{1\leq
    i\leq k}|x_i-f(t_i)|<\delta\}$. By Proposition
  \ref{prop:LDP-finite-dim}, we obtain that
  $$\liminf_{\eps\downarrow 0}\eps^{\alpha'}\log \Prob{\eps\Lex\in
    G}\geq -\inf_{(x_1,\ldots,x_k)\in G'}J_\sigma(x_1,\ldots,x_k)\,
  .$$
  Letting $\delta\to 0$, we may conclude that
  $$\liminf_{\eps\downarrow 0}\eps^{\alpha'}\log \Prob{\eps\Lex\in
    G}\geq -I^\sigma_\Lex(f)\geq -I_\Lex(f)\geq -I_\Lex(G)-\eps\,
  ,$$
  as desired. 
\end{proof}


\section{Consequences of the {\fontfamily{lmss}\selectfont LDP} for stable
  excursions}\label{sec:cons-ldp-stable}

In this section, we prove the remaining statements: 
Theorem \ref{thm:LDP-functional}, Corollaries \ref{cor:exact-asymp-Profeta}
and \ref{cor:exact-asymp-Profeta-sup}, and Proposition
\ref{thm:no-J1}.

\subsection{Proof of Theorem \ref{thm:LDP-functional}}\label{section:thm2}

We follow the approach of 
\cite{Fill-Janson}
closely. First, a direct consequence of the fact that $I_\Lex$ is a
good rate function is the following. 

\begin{lemma}\label{sec:proof-crefthm:ldp-fu}
 The set $\Kex$ defined at \eqref{eq:6} is a compact subset of $\D[0,1]$. 
 \end{lemma}

 The argument for the proof of Theorem \ref{thm:LDP-functional}  is the same
 as \cite[p.415]{Fill-Janson}. We apply the contraction principle
 (\cite[Theorem 27.11]{Kallenberg}, \cite[Theorem
 4.2.1]{Dembo-Zeitouni}) to the continuous functional $\Phi: \Dex[0,1]
 \rightarrow \R_+$. This entails that $\varepsilon X = \Phi( \varepsilon \Lex )$ satisfies an {\fontfamily{lmss}\selectfont LDP} in $[0,\infty)$ with good rate function whose value at $x>0$ is given by
 \begin{eqnarray*}
 \inf_{f \in \Hex \: : \: \Phi(f) = x} c_\alpha \Vert \fdo' \Vert^{\alpha'}_{\alpha'}
 & = & \inf_{f \in \Hex \: : \: \Phi(f) \neq 0} c_\alpha \left\Vert \frac{x \, \fdo'}{\Phi(f)}\right\Vert^{\alpha'}_{\alpha'} \\
 & = & \inf_{f \in \Hex \: : \: \Phi(f) \neq 0} c_\alpha \left( \frac{x}{\Phi(f)} \right)^{\alpha'}\|\fdo'\|_{\alpha'}^{\alpha'} \\
 & = & \inf_{f \in \Hex \: : \: \, \Phi(f) \neq 0} c_\alpha \left( \frac{x}{\Phi(f/\|\fdo\|_{\alpha'})} \right)^{\alpha'} \\
 & = & \inf_{f \in \Hex \: : \:  \|\fdo'\|_{\alpha'}=1,\Phi(f) \neq 0} c_\alpha \left( \frac{x}{\Phi(f)} \right)^{\alpha'} \\ 
 & = & c_\alpha \left(\frac{x}{\gamma_\Phi}\right)^{\alpha'} .
 \end{eqnarray*}
 Taking $A = (1,\infty)$ and $\varepsilon = x^{-1}$ in the definition
 of an {\fontfamily{lmss}\selectfont LDP} proves \eqref{eq:asymp-functional}. 
 Finally, \eqref{eq:asymp-generating-function} and \eqref{eq:asymp-moments} follow from \eqref{eq:asymp-functional} by \cite[Theorem 4.5]{Chassaing-Janson}.
 This concludes the proof of Theorem \ref{thm:LDP-functional}.
 \hfill
 \qedsymbol

\subsection{Applications}\label{section:app}

Corollaries \ref{cor:exact-asymp-Profeta}
and \ref{cor:exact-asymp-Profeta-sup} are obtained by applying Theorem 
\ref{thm:LDP-functional} to the area and supremum functionals, which
are both positive-homogeneous and continuous for the M1 topology.

 \subsubsection{Area under the normalized excursion}

 Let us compute the constant $\gamma_\Phi$ for the area under the normalized excursion $\Aex$, corresponding to the functional $\Phi(f) = \int_0^1 f(s) \d s$. So let
 $$ \gamma_{\int} = \max \bigg\{ \int_0^1 f(u) du : f\in \Kex \bigg\}. $$
 
 \begin{lemma}[Constant $\gamma_\Phi$ for $\Aex$]
 We have $\gamma_{\int} = (\alpha+1)^{-1/\alpha}$.
 \end{lemma}
 
 \begin{proof}
 We first find an upper bound. Let $f\in \Kex$. Note that,
 integrating by parts, we have 
 \begin{align*}
 \int_0^1 f(s) \d s
 & = \int_0^1 \fup(s) \d s - \int_0^1 \fdo(s) \d s \\
 & = \int_0^1 (\fup(s)-\fdo(1)) \d s + \int_0^1 s\fdo'(s) \d s\\
 & \leq \bigg( \int_0^1 |\fdo'(s)|^{\alpha'} \d s \bigg)^{1/\alpha'}
   \bigg(\int_0^1 s^{\alpha} \d s\bigg)^{1/\alpha} \le
   (\alpha+1)^{-1/\alpha}\, ,
 \end{align*}
 where in the third line we have used the fact that $\fup(1)-\fdo(1)=f(1)=0$, which entails that $\fup\leq \fdo(1)$, and then H\"older's inequality. We obtain $\gamma_{\int} \le (\alpha+1)^{-1/\alpha}$.
 
 Now note that $f(s) = \frac{(\alpha+1)^{1/\alpha'}}{\alpha} (1 -
 s^{\alpha})$ lies in $\Kex$ and is such that  
 $$ \int_0^1 f(s) \d s = (\alpha+1)^{-1/\alpha}, $$
 so that $\gamma_{\int} = (\alpha+1)^{-1/\alpha}$ is indeed the optimum.
 \end{proof}

 \subsubsection{Supremum of the normalized excursion}
 
 We now compute the constant $\gamma_\Phi$ corresponding to the
 functional $\sup_{0\le t\le 1} f(t)$ which is continuous for the M1
 Skorokhod topology.
 
 \begin{lemma}[Constant $\gamma_\Phi$ for $\sup\Lex$]
 We have $\gamma_{\sup} = 1$.
 \end{lemma}
 
 \begin{proof}
 First, notice that if $f\in\Kex$, then using the fact $\fup\leq \fdo(1)$ we get that for all $t\ge0$, 
 $$ f(t) \le \fdo(1) - \fdo(t) = \int_t^1 |\fdo'(s)|\d s \le \int_0^1 |\fdo'(s)|\d s \le \left( \int_0^1 |\fdo'(s)|^{\alpha'} \d s\right)^{1/\alpha'} = 1 , $$
 where we used H\"older's inequality. We thus obtain the upper bound $\gamma_{\sup}\le 1$. 
 
 Now note that the function $f(t) = 1-t$ lies in $\Kex$ and satisfies $\sup f = 1$, so that $\gamma_{\sup} =1$.
 \end{proof}

 \subsection{Negative results for the Skorokhod J1 topology}\label{subsection:no-J1}

 Here we give a sketch of proof for Proposition \ref{thm:no-J1}. The idea is that it is costless for
 the process $\eps\Lex$ to make macroscopic jumps within small time
 intervals, which prevents it from being concentrated in J1-compact
 sets. 
 Fix $\delta>0$. We note that
 \begin{align}
    \Prob{\eps\Lex_\delta\in [1,2],\eps\Lex_{2\delta}\in [3,4]}
  & = \alpha\Gamma\left(1-\frac{1}{\alpha}\right)\int_{1/\eps}^{2/\eps}\d
 x_1\, 
 j_{\delta}(x_1)\int_{3/\eps}^{4/\eps}\d x_2\,
 p_{\delta}^{(0,\infty)}(x_1,x_2)q_{x_2}(1-2\delta) \nonumber \\
& \geq
C(\delta)\eps^{2\alpha+2}\exp(-c_\alpha(3/\eps(1-2\delta))^{\alpha'})\, ,
\label{eq:23}
\end{align}
where we have used Lemmas \ref{lem:error-p_t}, 
\ref{lem:bound-error} and \eqref{eq:10} to bound
$p_{\delta}^{(0,\infty)}(x_1,x_2)$ uniformly from below by some constant times
$\eps^{\alpha+1}$, then \eqref{eq:11} to bound $q_{x_2}(1-2\delta)$
uniformly from below by some constant times
$\exp(-c_\alpha(3/\eps(1-2\delta))^{\alpha'})$, and finally
\eqref{eq:13} to bound the remaining integral.
Setting $\omega_{J1}(f,\eta)=\sup_{s<t<u,u-s\leq
  \eta}|f(u)-f(t)|\wedge |f(t)-f(s)|$, we obtain
$$\lim_{\delta\downarrow 0}\liminf_{\eps\downarrow
  0}\eps^{\alpha'}\log\Prob{\omega_{J1}(\eps\Lex,2\delta)> 1/2}
\geq -3^{\alpha'}c_\alpha\, .$$
Let $K$ be a compact subset of $\D[0,1]$ in the J1 topology, so that $\sup_{f\in K}\omega_{J1}(f,\delta)$ converges to $0$ as $\delta\downarrow 0$. In particular, there exists $\delta_0$ such that $\omega_{J1}(f,2\delta_0)\leq 1/2$ for every $f\in K$. Therefore, 
$$\liminf_{\eps\downarrow 0}\eps^{\alpha'}\log\Prob{\eps\Lex\notin K}
\geq \liminf_{\eps\downarrow 0}\eps^{\alpha'}\log\Prob{\omega_{J1}(\eps\Lex,2\delta_0)> 1/2}
\geq -3^{\alpha'}c_\alpha\, ,$$
and so $(\eps\Lex)_{0<\eps<1}$ cannot be exponentially tight.


\section{{\fontfamily{lmss}\selectfont LDP} for the $\alpha$-stable L\'evy bridge}\label{section:bridge}

 In this section, we adapt the proof of the
 {\fontfamily{lmss}\selectfont LDP} for the normalized excursion in
 order to get an {\fontfamily{lmss}\selectfont LDP} for the L\'evy bridge.
 Roughly speaking the process $\bra$, called the $(0,0) \rightarrow (1,a)$ bridge, is obtained by conditioning $L$ to be equal to $a$ at time $1$.
 This is obviously a degenerate conditioning; however, it can be
 obtained by performing a space-time $h$-transform with respect to the function $\frac{p_{1-t}(a-L_t)}{p_1(a)}$ (see, for instance, \cite{Liggett}). This means that the law of $\bra$ may be defined by 
 \begin{equation}
   \label{eq:20}
     \bP^{\textrm{br}}(A) =
     \E{\frac{p_{1-t}(a-L_t)}{p_1(a)} \mathbbm 1_A}, \quad
     \forall A \in \Fil_t\, ,\quad t\in [0,1) .
   \end{equation}
   See \cite{Chaumont} or \cite[Chapter VIII]{Bertoin} for a rigorous construction.

 \subsection{Large deviations for the finite-dimensional marginal distributions}
 
  This section is devoted to proving that the finite-dimensional marginals of $\bra$ satisfy an {\fontfamily{lmss}\selectfont LDP} on $\R$.
  
  \begin{proposition}[{\fontfamily{lmss}\selectfont LDP} for the marginals of the stable bridge]\label{prop:LDP-marginals-bridge}
  Fix $a\in \R$, and let $(a_\eps)_{\eps>0}$ be such that $\eps a_\eps \to a$ as
  $\eps\to0$. Let $\sigma=(t_1,\ldots,t_n)$ be a finite subdivision of
  $[0,1]$. Under $\bP$, the laws of
  $\eps(\brae_{t_1},\dots,\brae_{t_n})$ satisfy an {\fontfamily{lmss}\selectfont LDP} in $\R^n$ with
  speed $\eps^{-\alpha'}$ and good rate function 
  $$ J_{\br,a}^\sigma(x_1,\dots,x_n)
  = c_\alpha \left( \left( \frac{(-x_1)_+^\alpha}{t_1} \right)^{\frac{1}{\alpha-1}} + \sum_{i=1}^{n-1} \left(\frac{(x_i-x_{i+1})_+^\alpha}{t_{i+1}-t_i}\right)^{\frac{1}{\alpha-1}} + \left( \frac{(x_n-a)_+^\alpha}{1-t_n} \right)^{\frac{1}{\alpha-1}}  - (a_-)^{\alpha'} \right) . $$
  \end{proposition}

The proof of this proposition is similar to that of Proposition
\ref{prop:LDP-finite-dim}, but is technically simpler and we only
explain where the argument differs.
It is not difficult to check that $J_{\br,a}$ is a good rate
  function, so \cite[Lemma 5]{Serlet} applies. 
We use this in the same way as in the proof of Proposition \ref{prop:LDP-finite-dim}, using the expression \eqref{eq:8}  for the marginal laws of the
bridge, and the 
  bounds \eqref{eq:11} for the stable transition densities. The term
  $-(a_-)^{\alpha'}$ in the definition of $J_{\br,a}$ arises from
  the contribution of the density $p_1(a_\eps)$ in the denominator of \eqref{eq:8}: 
  by \eqref{eq:11} and \eqref{eq:12}, we have 
$$ \lim_{\eps\downarrow0} \eps^{\alpha'}\log p_1(a_\eps)
  = -c_\alpha(a_-)^{\alpha'}.$$

 \subsection{Exponential tightness for the stable bridge}
 
  \begin{proposition}[Exponential tightness for the stable bridge]\label{prop:exponential-tightness-bridge}
  Under $\bP$, the laws of $(\eps \brae_t)_{t\in[0,1]}$ as $\eps\downarrow0$ are exponentially tight with speed $\eps^{-\alpha'}$.
  \end{proposition}
  
  Proposition \ref{prop:exponential-tightness-bridge} is a direct consequence of
  the tightness criterion in $(\D[0,1],\mathrm{dist})$ from Theorem \ref{sec:basic-results-skor-2} and the following lemma.
  
  \begin{lemma}\label{lem:exponential-tightness-bridge}
 There exists a constant $C=C(\alpha)>0$ such that for every $s,t,u\in [0,1]$ with $s\leq t\leq u$ and $\lambda\geq 0$,
  \begin{equation}
  \E{\exp\Big(\lambda M(\brae_s,\brae_t,\brae_u)\Big)} \leq C \exp\big((u-s)\lambda^\alpha\big) .
  \end{equation}
  \end{lemma}
  
  \begin{proof}
  Splitting the expectation into five terms as at the beginning of the
  proof of Proposition \ref{prop:tightness}, we get the following bound
  \begin{multline*}
           \E{\exp\big(\lambda M(\brae_s,\brae_t,\brae_u)\big)} \\
           \leq 2 \E{e^{\lambda(\brae_s-\brae_t)} \1{\brae_t\leq \brae_s}} 
           + 2 \E{e^{\lambda(\brae_t-\brae_u)} \1{\brae_u\leq
               \brae_t}} + 1  .
         \end{multline*}
         We see that the last two terms are of the same form $\bE\Big[e^{\lambda (\brae_\sigma-\brae_\rho)} \1{\brae_\rho\leq \brae_\sigma}\Big]$ where $\rho\leq \sigma$ with $\sigma-\rho\leq u-s$.
  For such $\rho,\sigma$, we have
  \begin{align*}
  \E{e^{\lambda (\brae_\sigma-\brae_\rho)} \1{\brae_\rho\leq \brae_\sigma}}
  & = \int_{-\infty}^{+\infty} \d x \int_{-\infty}^{+\infty} \d y\, p_\rho(x) p_{\sigma-\rho}(y-x) \frac{p_{1-\sigma}(a_\eps-y)}{p_1(a_\eps)} e^{\lambda(x-y)} \1{y\leq x} \\
  & = \frac{1}{p_1(a_\eps)} \int_0^\infty \d z\,  p_{\sigma-\rho}(a_\eps-z) e^{\lambda z} \int_{-\infty}^{+\infty} \d x\, p_\rho(x) p_{1-\sigma}(z-x) ,
  \end{align*}
  where we have used the change of variables $z=x-y+a_\eps$. 
  It remains to show that the last integral in $x$ is uniformly
  bounded over $z\geq0$, $0\leq \rho < \sigma \leq 1$.
  Indeed this gives the existence of a constant $C<\infty$ such that
  $$ \E{e^{\lambda (\brae_\sigma-\brae_\rho)} \1{\brae_\rho\leq \brae_\sigma}} \leq C \E{e^{-\lambda \brae_{\sigma-\rho}}} = C \exp\big((\sigma-\rho)\lambda^\alpha\big) , $$
  which gives the result.

  However, by the
 cyclic invariance of the increments of $\bra$, which is a direct
 consequence of \eqref{eq:8}, we see that it suffices to check this
 boundedness assumption for $0\leq \rho<\sigma\leq 1/2$ say. 
 For such $\sigma,\rho$ we have 
  \begin{align*}
  \int_{-\infty}^{+\infty} \d x\, p_\rho(x) p_{1-\sigma}(z-x) 
  & = (1-\sigma)^{-\frac{1}{\alpha}} \int_{-\infty}^{+\infty} \d x\, p_\rho(x) p_1\left(\frac{z-x}{(1-\sigma)^{\frac{1}{\alpha}}}\right) \\
  & \leq 2^{\frac{1}{\alpha}} \Vert p_1 \Vert_\infty \int_{-\infty}^{+\infty} \d x\, p_\rho(x) \\
  & = 2^{\frac{1}{\alpha}} \Vert p_1\Vert_\infty \int_{-\infty}^{+\infty} \d x\, \rho^{-\frac{1}{\alpha}} p_1\left(\frac{x}{\rho^{\frac{1}{\alpha}}}\right) \\
  & = 2^{\frac{1}{\alpha}} \Vert p_1 \Vert_\infty \int_{-\infty}^{+\infty} \d x\, p_1(x) ,
  \end{align*}
  where the last integral does not depend on $\rho$. 
  \end{proof}

 \subsection{Proof of Theorem \ref{thm:LDP-bridge}}
 
 We may finally prove Theorem \ref{thm:LDP-bridge} using
 Propositions \ref{prop:LDP-marginals-bridge} and
 \ref{prop:exponential-tightness-bridge}. The scheme of proof is
 exactly the same as that in Section \ref{section:thm1}, and combines the exponential tightness in
 $(\D[0,1],\dist)$ with an {\fontfamily{lmss}\selectfont LDP} in the weak topology
 $(\D[0,1],\mathcal{W})$, similar to Proposition
 \ref{sec:proof-crefthm:ldp-ex}. Therefore, we will  give a brief
 account,  only pointing out the places where the formulas differ. 

 We can easily adapt the proof of Proposition \ref{prop:criteria-Hex} to get the
 following proposition. For a subdivision $\sigma=(t_1,\ldots,t_n)$
 and $f\in\D[0,1]$, 
 define $I_{\br,a}^\sigma(f)=J_{\br,a}^\sigma(f(t_1),\ldots,f(t_n))$. 
 
 \begin{proposition}
 A function $f\in \D[0,1]$ is in $\Hbra$ if and only if
 $$ M_{\br,a}(f) = \sup_{\sigma \in \mathfrak S} 
I_{\br,a}^\sigma(f)
 < \infty . $$
 In this case, we have
 $$M_{\br,a}(f) =I_{\br,a}(f).$$
 \end{proposition}

Then, analogs of Lemmas \ref{sec:facts-about-rate-1},
\ref{sec:facts-about-rate-2}, \ref{sec:facts-about-rate-3} and
\ref{sec:facts-about-rate-4} hold true with $I_{\br,a}$ in place of
$I_\Lex$,  with exactly the same proofs. This ends the proof of Theorem 
\ref{thm:LDP-bridge}.
 
Theorem \ref{thm:LDP-functional-bridge} can then be deduced from Theorem
\ref{thm:LDP-bridge} along the same lines as in Section
\ref{section:app}.

\section*{Acknowledgements}
 Thanks are due to Loïc Chaumont for an  interesting conversation
 around stable excursions. We also thank the anonymous referee for
 their careful reading of the paper, and for bringing a number of
 relevant references to our attention. 

\bibliographystyle{alea3}
\bibliography{bib-LDPstable}

\end{document}